\documentclass[hidelinks]{article}
\usepackage[utf8]{inputenc}
\usepackage[utf8]{inputenc}
\usepackage{mathrsfs}
\usepackage{amssymb}
\usepackage[margin=1.1in]{geometry}
\usepackage{changepage}

\usepackage{xcolor}
\usepackage{bbold}
\usepackage{amsmath,amsthm}
\usepackage{graphicx}
\usepackage{hyperref}

\newtheorem{definition}{Definition}
\newtheorem{propo}{Proposition}
\newtheorem{remark}{Remark}
\newtheorem{lemma}{Lemma}

\newtheorem{theorem}{Theorem}
\newtheorem*{theorem*}{Theorem}

\numberwithin{equation}{section}
\newcommand{\vertiii}[1]{{\left\vert\kern-0.25ex\left\vert\kern-0.25ex\left\vert #1 
    \right\vert\kern-0.25ex\right\vert\kern-0.25ex\right\vert}}

\def\N{\mathbb{N}}
\def\R{\mathbb{R}}

\def\d{\mathrm{d}}
\def\B{\mathbb{B}}
\def\P{\mathbb{P}}
\def\ssi{\Longleftrightarrow}

\def\eps{\varepsilon}

\title{\textbf{Heat kernel estimates for stable-driven SDEs with distributional drift}}

\author{
Mathis Fitoussi\footnotemark[1]}

\date{\today}

\begin{document}

\maketitle
\footnotetext[1]{Universit\'e Paris-Saclay, Laboratoire de Math\'ematiques et Mod\'elisation d’\'Evry (LaMME), 23 boulevard de France, 91 037 \'Evry, France\\ \textit{Email address:} mathis dot fitoussi at universite-paris-saclay dot fr}

\begin{center}
\textbf{Abstract}
\end{center}
\begin{adjustwidth}{0.8in}{0.8in}
We consider the \textit{formal} SDE
\begin{equation}\label{sde-abstract}\tag{E}
	\d X_t = b(t,X_t)\d t + \d Z_t, \qquad X_0 = x \in \R^d,
\end{equation}
where $b\in L^r ([0,T],\B_{p,q}^\beta (\R^d,\R^d))$ is a time-inhomogeneous Besov drift and $Z_t$ is a symmetric $d$-dimensional $\alpha$-stable process, $\alpha \in (1,2)$, whose spectral measure is absolutely continuous w.r.t. the Lebesgue measure on the sphere. Above, $L^r$ and $\B_{p,q}^\beta$ respectively denote Lebesgue and Besov spaces. We show that, when $\beta > \frac{1-\alpha + \frac{\alpha}{r} + \frac{d}{p}}{2}$, the martingale solution associated with the formal generator of \eqref{sde-abstract} admits a density which enjoys two-sided heat kernel bounds as well as gradient estimates w.r.t. the backward variable. Our proof relies on a suitable mollification of the singular drift aimed at using a Duhamel-type expansion. We then use a normalization method combined with Besov space properties (thermic characterization, duality and product rules) to derive estimates.
\end{adjustwidth}

\begin{paragraph}{Keywords} Heat kernel estimates, singular drift diffusions, stable SDEs.
\end{paragraph}
\begin{paragraph}{Classification (MSC 2020)} 35K08 (Heat Kernel), 60H10 (Stochastic Ordinary Differential Equations), 60G52 (Stable Stochastic Processes), 60H50 (Regularization by Noise)
\end{paragraph}

\section{Introduction and notations}\label{sec-intro}
Throughout this paper, the dimension of the problem $d$ and $\alpha \in (1,2)$ are fixed.
\subsection{Scope of this paper}\label{subsec-scope}
For a fixed $T>0$, we study the \textit{formal} SDE
\begin{equation}\label{sde}
\d X_t = b(t,X_t)\d t + \d Z_t, \qquad X_0 = x, \qquad \forall t \in [0,T],
\end{equation}
where $b\in L^r ([0,T],\B_{p,q}^\beta (\R^d,\R^d))= \left\{ f :[0,T] \times \R^d : \left\Vert t\mapsto \Vert f(t,\cdot) \Vert_{\B_{p,q}^\beta} \right\Vert_{L^r([0,T])}  < \infty \right\}$ (see Subsection \ref{subsec-besov} for details on Besov spaces) and $Z_t$ is a symmetric non-degenerate $d$-dimensional $\alpha$-stable process, whose spectral measure is absolutely continuous w.r.t. the Lebesgue measure on $\mathbb{S}^{d-1}$ (see Subsection \ref{subsec-density-noise}  for detailed assumptions on the noise).\\

We call (\ref{sde}) \textit{``formal"} equation because $b$ can be a distribution when $\beta<0$, in which case (\ref{sde}) is ill-defined as such. As there are multiple ways to define a solution to (\ref{sde}), each with its conditions on the parameters $\beta,p,q,r$ and interpretation, we will go into details in Subsection \ref{subsec-def}. \\

The main idea behind the study of singular drift diffusions is that adding a noise regularizes ordinary differential equations, and helps restore existence and uniqueness in some appropriate sense. For example, in the case of a $\beta$-Hölder ($\beta\in(0,1)$) drift, the noise gives an ``impulse" which permits to exit singular spots (see e.g. \cite{DF14} in the Brownian case). Knowing that, one would expect that, the bigger the intensity of the noise, the stronger the regularizing effect, which we will see on the upcoming thresholds (see also \cite{CdRM22b} and \cite{MM21}). We will investigate cases in which the noise is strong enough to restore uniqueness even for distributional drifts.\\

Let us first review the probabilistic results and associated techniques used in the case $\beta\geq 0, \alpha \in (0,2)$ to derive weak or strong well-posedness, when the drift is a function. In order to establish well-posedness, a natural condition appeared in the seminal article by Tanaka \textit{et al.} \cite{TTW74} : $\beta + \alpha >1$. The authors consider therein the scalar case, and proved that strong uniqueness holds for bounded $\beta$-Hölder drifts under this condition, while giving a counter example when $\beta+\alpha <1$. The critical multidimensional case (i.e. $\alpha =1$) in a time-inhomogeneous setting was investigated in \cite{Kom84}, in which weak uniqueness is derived for a continuous drift with, again, the driving noise having absolutely continuous spectral measure w.r.t. the Lebesgue measure on the sphere. Having in mind that weak (or strong) well-posedness is often investigated through the corresponding parabolic PDE, recalling that the associated expected parabolic gain is $\beta+\alpha$, the  condition $\beta+\alpha >1$ coincides with the regularity required to define the gradient of the solution. The aforementioned regularity gain is often obtained through Schauder-type estimates. We can mention \cite{MP14} (bounded drift, stable-like generators), \cite{CdRMP20} (unbounded drift, general stable generators including e.g. the cylindrical one). These estimates naturally lead to weak uniqueness in the multidimensional setting for \eqref{sde} through the martingale problem, which precisely requires a control of the gradient of the solution of the PDE.\\

Going towards strong solutions requires additional constraints on the parameters. It was e.g. shown by Priola in \cite{Pri12} that pathwise uniqueness holds in the multidimensional case for general non-degenerate stable generators with $\alpha \geq 1$ for time-homogeneous bounded $\beta$-Hölder drifts under the assumption $\beta > 1- \alpha/2$. Under the same assumption, \cite{CZZ21} proved strong existence and uniqueness for any $\alpha \in (0,2)$, as well as weak uniqueness whenever $\beta + \alpha >1$ for time-inhomogeneous drift with non-trivial diffusion coefficient. Those results are usually obtained using the Zvonkin transform (see \cite{Zvo74}, \cite{Ver80}), which requires additional regularity on the underlying PDE, which again follow from Schauder-type estimates.\\

Once weak or strong well-posedness is established, a natural question concerns the behavior of the time marginal laws of the SDE. Such behavior is usually investigated through heat kernel estimates, which, in the stable setting, somehow forces to consider the stable-like case, i.e., the driving noise $Z$ in \eqref{sde} has a L\'evy measure with \textit{smooth} density w.r.t. to the isotropic $\alpha$-stable measure (see Subsection \ref{subsec-density-noise} for detailed assumptions on the noise). In this setting, we can refer to the seminal work by Kolokoltsov \cite{Kol00}, who addressed the subcritical case $\alpha > 1$ for smooth bounded drifts. This work was extended in various directions, although mostly for non-negative $\beta$ (see \cite{Kul19}, \cite{CHZ20}, \cite{KK18}).
   In \cite{MZ22}, authors cover the whole range $\alpha \in (0,2)$ with Hölder unbounded drift. In those works, the authors establish that the time marginal laws of the process have a density which is ``equivalent" (i.e. bounded from above and below) to the density of the noise, and that the spatial gradients exhibit the same time singularities and decay rates (see Theorem \ref{thm-main} below in the current setting).\\

Going towards negative $\beta$ brings additional difficulties.  The first challenge is to specify what is intended with ``solution" to (\ref{sde}). To this end, a key tool is the following PDE:
\begin{equation}\label{pde}
	\left( \partial_t + b\cdot D + \mathcal{L}^\alpha \right)u(t,x)=f(t,x) \; \mathrm{on} \;[0,T)\times \R^d, \qquad u(T,\cdot) = g \; \mathrm{on} \; \R^d
\end{equation}
 for suitable sources $f$ and final conditions $g$, and where $\mathcal{L}^\alpha$ is the generator of the noise $Z$. When studying \eqref{pde}, defining the gradient of the solution still requires $\alpha + \beta>1$, which now imposes $\alpha >1$. This is anyhow not sufficient: we also need to be able to define $b\cdot Du$ as a distribution. Roughly, since $b$ has spatial regularity $\beta$, this imposes $\beta + (\beta + \alpha - 1 )> 0\ssi \beta > \frac{1-\alpha}{2}$ by usual paraproduct rules (note that this is the exact assumption we need if $p=r=+\infty$). This threshold already appears in \cite{BC01} in the diffusive setting ($\alpha=2$), where strong well-posedness is derived in the scalar case through Dirichlet forms techniques for specifically structured time-homogeneous drifts. The same threshold is exhibited in \cite{FIR17}, where the authors introduce the notion of \textit{virtual solutions} to give  a meaning to \eqref{sde}. Those solutions are defined through a Zvonkin-type transform formula, and, while not requiring any specific structure, do not yield a precise dynamics for the SDE. We can also refer to \cite{ZZ17} and \cite{ABM20}, who specified the meaning to be given to \eqref{sde}, in the sense that the drift therein is defined through smooth approximating sequences of the singular $b$ along the solution. Importantly, the limit drift is a Dirichlet process, highlighting once again that \eqref{sde} is a \textit{formal} equation. A thorough description of this Dirichlet process was done in the Brownian scalar case in \cite{DD16} and extended in \cite{CC18} for multidimensional SDEs. Assuming some additional structure on the drift, they manage to go beyond the above threshold and reach $\beta  > -\frac{2}{3}$ (still with $p=r=\infty$). This work was extended in the multidimensional strictly stable case, still assuming a specific structure for the drift in \cite{KP22}, in which weak well-posedness is proved for $\beta > \frac{2-2\alpha}{3}$. Without any structure on the drift, a similar and consistent description of the dynamics for the weak solutions of \eqref{sde} in the multidimensional setting is obtained in \cite{CdRM22} for $\beta> \frac{1-\alpha}{2}$. The case of a non-trivial diffusion coefficient was investigated in \cite{LZ22} with the same thresholds. Note that, in the present work, we chose to work with a trivial diffusion coefficient as the most delicate issue is the handling of the singular drift (see Remark 15 in \cite{CdRM22} for the handling of a non trivial diffusion coefficient in a Duhamel expansion). We do believe our approach would be robust enough to treat this case. Let us mention that \cite{ABM20} also obtained strong uniqueness with this threshold in the scalar case. We emphasize that most of the aforementioned results heavily rely on the Schauder-type regularization properties of the PDE \eqref{pde}.\\

In the scope of singular drift heat kernels estimates, the sole result we were able to gather is \cite{PvZ22}. Using the Littlewood-Paley characterization of Besov spaces, Perkowski and van Zuijlen managed to derive explicit two-sided heat kernel estimates as well as gradient estimates w.r.t. the backward variable for the solution in the Brownian, time-inhomogeneous setting with time-continuous drift in $\B^\beta_{\infty,1},\beta>-\frac{1}{2}$.\\

The goal of the current paper is to establish heat kernel and gradient estimates for stable driven SDEs with drifts in $L^r-\B_{p,q}^\beta$ and symmetric non-degenerate $d$-dimensional $\alpha$-stable noise with absolutely continuous L\' evy measure for $ \beta \in \left( \frac{1-\alpha+\frac{d}{p}+\frac{\alpha}{r}}{2} ,0\right)$. As compared to the previous results, this represents a slight modification of the threshold, due to integrability concerns. For $p=r=+\infty$, we work under the usual $\beta > \frac{1-\alpha}{2}$ assumption.\\

This paper is organized as follows. We first discuss the properties of the noise in Subsection \ref{subsec-density-noise}. We then  define the notions of martingale solutions for \eqref{sde} and mild solutions for \eqref{pde} along with required assumptions in Subsection \ref{subsec-def}. We state our main results in Subsection \ref{subsec-main-thm} and detail the dynamics of \eqref{sde} in \ref{subsec-weak-sols}. Section \ref{sec-lemmas} is a collection of technical lemmas specific to our paper and classical results in Besov spaces, which we use in the following Sections \ref{sec-estimates} and \ref{sec-approx}. Section \ref{sec-estimates} is dedicated to obtaining estimates on a mollified equation with smooth drift and Section \ref{sec-approx} links those estimates back to the main SDE (\ref{sde}) through compactness arguments. Section \ref{sec-proof} contains the proofs of all technical lemmas.

\subsection{Driving noise and related density properties}\label{subsec-density-noise}


Let us denote by $\mathcal{L}^\alpha$ the generator of the driving noise $Z$. In the case $\alpha = 2$, $\mathcal{L}^\alpha$ is the usual Laplacian $\frac{1}{2}\Delta$. When $\alpha \in (1,2)$, in whole generality, the generator of a symmetric stable process writes, $\forall \phi \in C_0^\infty (\R^d,\R)$ (smooth compactly supported functions),
\begin{align*}
\mathcal{L}^\alpha \phi (x)&= \mathrm{p.v.} \int_{\R^d} \left[ \phi(x+z) - \phi(x)\right]\nu(\d z)\\
&=\mathrm{p.v.}\int_{\R_+}\int_{\mathbb{S}^{d-1}}\left[ \phi(x+\rho \xi) - \phi(x)\right]\mu(\d \xi)\frac{\d \rho}{\rho^{1+\alpha}}
\end{align*}
(see \cite{Sat99} for the polar decomposition of the spectral measure) where $\mu$ is a non-degenerate measure on the unit sphere $\mathbb{S}^{d-1}$, i.e. $\mu$ is symmetric and $\exists \kappa \geq 1 : \forall \lambda \in \R^d$,
\begin{equation*}
\kappa^{-1} |\lambda|^\alpha \leq \int_{\mathbb{S}^{d-1}} |\lambda \cdot \xi|^\alpha \mu(\d \xi) \leq \kappa |\lambda|^\alpha,
\end{equation*} 
where ``$\cdot$" stands for the usual scalar product in $\R^d$.\\
This general setting will not allow us to derive heat kernel estimates, because it does not lead to global estimates of the noise density. In \cite{Wa07}, Watanabe investigates the behavior of the density of an $\alpha$-stable process in terms of properties fulfilled by the support of its spectral measure. From this work, we know that whenever the measure $\mu$ is not equivalent to the Lebesgue measure on the unit sphere, accurate estimates on the density of the stable process are delicate to obtain. However, Watanabe (see \cite{Wa07}, Theorem 1.5) and Kolokoltsov (\cite{Kol00}, Propositions 2.1--2.5) showed that if $C^{-1} m(\d \xi) \leq \mu (\d \xi) \leq C m(\d \xi)$ (where $m$ is the uniform density on $\mathbb{S}^{d-1}$), the following estimates hold: there exists a constant $C$ depending only on $\alpha,d$, s.t. $\forall u\in \R_+^*, z\in \R^d$,
$$\frac{C^{-1}}{u^{\frac{d}{\alpha}}} \frac{1}{\left( 1+ \frac{|z|}{u^{\frac{1}{\alpha}}} \right)^{d+\alpha}} \leq p_\alpha (u,z) \leq \frac{C}{u^{\frac{d}{\alpha}}} \frac{1}{\left( 1+ \frac{|z|}{u^{\frac{1}{\alpha}}} \right)^{d+\alpha}}.$$

As our approach heavily relies on these global bounds, we have to assume that $\mu$ is equivalent to the Lebesgue measure on the sphere and that $\alpha \in (1,2)$.\\

Further properties related to the density of the driving noise are stated in Lemma \ref{pre-big-lemma} below.

\subsection{Notations and definitions}\label{subsec-def}
We will use the following notations : 
\begin{itemize}
\item The set of all parameters will be denoted $\Theta := \{\alpha,d,\beta,r,p,q,\Vert b \Vert_{L^{r} -\B^\beta_{p,q}}\}$
\item $a\lesssim b $ if there exists a constant $C$, which depends only on parameters from $\Theta$, such that $a\leq Cb$.
\item $a \asymp b$ if there exists a constant $C$, which depends only on parameters from $\Theta$, such that $ C^{-1} b \leq a\leq Cb$.
\item $\star$ denotes the spatial convolution.
\item $C^{0,1}([0,T]\times \R^d,\R)$ is the space of continuous in time and differentiable in space functions, $C_b^{0,1}([0,T]\times \R^d,\R)=C^{0,1}([0,T]\times \R^d,\R)\cap L^\infty([0,T]\times \R^d,\R)$ and $C_b^0 ([0,T]\times \R^d,\R)$ is the space of bounded continuous space-time functions.
\item For $f\in\mathcal{S}'(\R^d)$ (the dual of the Schwartz class $\mathcal{S}(\R^d)$) and $\phi \in C_0^\infty (\R^d)$ such that $\phi(0) \neq 0$, we set $\phi(D)f = \mathcal{F}^{-1} \left( \phi \times \mathcal{F}(f) \right) = \mathcal{F}^{-1}(\phi)\star f$, where $\mathcal{F}$ denotes the Fourier transform.
\item For $p\in [1,+\infty]$, we always denote by $p'\in [1,+\infty]$ s.t. $\frac{1}{p}+ \frac{1}{p'}=1$ its conjugate.\\
\end{itemize}

For the rest of this paper, we will denote by $\bar{p}_\alpha$ the following density :
\begin{equation}
\bar{p}_\alpha (v,z) = \frac{C_\alpha}{v^{\frac{d}{\alpha}}} \frac{1}{\left(1+\frac{|z|}{v^{\frac{1}{\alpha}}} \right)^{d+\alpha}},\qquad v>0,z\in \R^d,
\end{equation}
where $C_\alpha$ is chosen so that $\forall v>0, \int \bar{p}_\alpha (v,y) \d y = 1$.\\

We finally introduce the semi-group generated by $\mathcal{L}^\alpha$: for any bounded Borel function $\phi$,
\begin{equation}
P_t^\alpha [\phi](x) := \int_{\R^d}p_\alpha (t,y-x) \phi(y) \d y.
\end{equation}

\begin{paragraph}{Martingale solutions\\[0.5cm]}

As we work with a distributional drift, we need to specify what we call a ``solution" to (\ref{sde}). There are two ways to define a solution to (\ref{sde}) which we will investigate. We will first introduce the usual \textit{martingale solutions}. Those are defined through the mild solutions of the underlying PDE and are the ones that require the least regularity. Importantly, they are sufficient to state Theorem \ref{thm-main}. In Subsection \ref{subsec-weak-sols}, we will then give details about \textit{weak solutions}, as defined in \cite{CdRM22} in order to give a concrete dynamics for the solution.\\

Although our results are proved for martingale solutions (in which case they can be understood as a \textit{formal} discussion on the density of the process), they are mainly useful in the scope of weak solutions, as those introduce a dynamics and could be a starting point to establish numerical schemes for those equations.\\

Let us now fix $p,q,r \geq 1$. For the definition of a \textbf{martingale} solution to (\ref{sde}), we need the following conditions on $\alpha,\beta$, which we call a \textbf{good relation (GR)} :
\begin{equation}\label{GR}\tag{GR}
\alpha \in \left( \frac{1+\frac{d}{p}}{1-\frac{1}{r}},2 \right)\qquad \beta \in \left( \frac{1-\alpha+\frac{d}{p}+\frac{\alpha}{r}}{2} ,0\right)
\end{equation}
and we will denote 
\begin{equation}\label{theta}
\theta :=\beta  + \alpha -\frac{d}{p} - \frac{\alpha}{r},
\end{equation}
which corresponds to the parabolic bootstrap induced by the drift.
As explained in \cite{CdRM22}, this choice of $\theta$ implies that
$$(t,x) \mapsto \int_t^T P_{s-t}^\alpha [G\cdot v ](s,x) \d s $$ is well defined and belongs to $\mathcal{C}^{0,1}_b([0,T]\times \R^d,\R)$ as soon as $G\in L^r ([0,T],\B_{p,q}^\beta (\R^d,\R^d))$ and $v\in L^{\infty}([0,T],\B_{\infty,\infty}^{\theta-1-\eps}(\R^d,\R^d))$ for some $0\leq \eps \ll 1$.\\
\begin{remark}  Note that, here, we are only trying to give a meaning to the distributional product $G \cdot v$. Roughly speaking, for $p=r=+\infty$, by Bony's paraproduct rule, the total regularity of $G\cdot v$ is $\beta +\theta - 1 - \eps$, which we need to be positive. This is only possible if $\alpha$ and $\beta$ satisfy \textbf{(\ref{GR})}, hence the definition of the latter. The additional $\frac{d}{p}+\frac{\alpha}{r}$ corresponds to the lack of global boundedness of the drift $b$.\\
\end{remark}

This allows us to give the definition of mild solution to a PDE:\\

\begin{definition}{Mild solution of the underlying PDE. \\} Let $\alpha \in (1,2)$, $f: \R_+\times \R^d \rightarrow \R$ and $g : \R^d \rightarrow \R$. For a given $T>0$, we say that $u:[0,T]\times \R^d \rightarrow \R$ is a \textbf{mild} solution of the formal Cauchy problem $\mathcal{C}(b,\mathcal{L}^\alpha,f,g,T)$
\begin{equation*}
\left( \partial_t + b\cdot D + \mathcal{L}^\alpha \right)u(t,x)=f(t,x) \; \mathrm{on} \;[0,T)\times \R^d, \qquad u(T,\cdot) = g \; \mathrm{on} \; \R^d,
\end{equation*}
if it belongs to $\mathcal{C}^{0,1}([0,T]\times \R^d,\R)$ with $Du\in \mathcal{C}_b^0([0,T],\B_{\infty,\infty}^{\theta -1-\eps})$ for any $0<\eps \ll 1$ and $\theta = \beta  + \alpha -\frac{d}{p} - \frac{\alpha}{r}$, and if it satisfies 
\begin{equation}\label{mild-sol}
\forall (t,x) \in [0,T]\times \R^d, u(t,x)=P_{T-t}^\alpha [g](x) - \int_t^T P_{s-t}^\alpha [f-b\cdot Du ](s,x) \d s.
\end{equation}

\end{definition}


In \cite{CdRM22}, Chaudru de Raynal and Menozzi proved existence and uniqueness of such solutions under \textbf{\eqref{GR}}, and also give information on their time regularity. Let us now introduce the notion of martingale problem (introduced in \cite{SV97} and then generalized in \cite{EK86}).

\begin{definition}\label{mart-sol}{Solution of the martingale problem\\} Let $\Omega = \mathcal D([0,T],\R^d)$ (the Skorokhod space of càdlàg functions). We say that a probability measure $\mathbb P$ on $\Omega$ equipped with its canonical filtration is a solution of the martingale problem associated with $(b,\mathcal{L}^\alpha,x)$ for $x \in \R^d$ if, denoting by $(x_t)_{t\in [0,T]}$ the associated canonical process,    
\begin{enumerate}
\item[(i)] $\displaystyle \mathbb P(x_0 = x) = 1$,
\item[(ii)] $\displaystyle \forall f \in \mathcal{ C}([0,T], \mathcal S(\R^d,\R))$,  $g \in \mathcal{C}^1(\R^d,\R)$ with $Dg \in \B_{\infty,\infty}^{\theta-1}(\R^d,\R^d)$,
$$\left(u(t,x_t) - \int_0^t f(s,x_s) ds - u(0,x)\right)_{0\le t \le T}$$ 
is a  martingale under $\mathbb{ P}$ where $u$ is the \textbf{mild} solution of the Cauchy problem $\mathcal{C}(b, \mathcal{L}^\alpha,f,g,T)$.
\end{enumerate}
\end{definition}
\begin{remark} The choice of the class of $f$ (here, $\mathcal{ C}([0,T], \mathcal S(\R^d,\R))$) is not critical. We only need it to be rich enough to characterize marginal laws, i.e. a class of functions $\Phi$ is sufficient if whenever two probability measures $\mu_1$ and $\mu_2$ satisfy $$\int \phi \d \mu_1 = \int \phi \d \mu_2, \qquad \forall \phi \in \Phi,$$ then $\mu_1 = \mu_2 $.
\end{remark}

Again, in \cite{CdRM22}, it is proved that there exists a unique solution to the martingale problem in the sense of the previous definition. We will call \textit{``martingale solution to \eqref{sde}"} the associated canonical process.

\end{paragraph}

\section{Main results}\label{sec-result}
\subsection{Main theorem}\label{subsec-main-thm}
\begin{theorem}\label{thm-main} Fix the parameters $T>0$ and $\Theta = \{\alpha,d,\beta,r,p,q,\Vert b \Vert_{L^{r} -\B^\beta_{p,q}}\}$. Take $b\in L^r ([0,T],\B_{p,q}^\beta ( \R^d))$ and assume  \textbf{(\ref{GR})} holds. Consider the solution $\mathbb{ P}$ to the martingale problem associated with $(b,\mathcal{L}^\alpha,x)$ starting at time $s$ and denote $(x_t)_{t\in {[s,T]}}$ the associated canonical process. For all $t\in (s,T]$, $x_t$ admits a density $p(s,t,x,\cdot)$, i.e. for all $A\in \mathcal B (\R^d)$ (Borel $\sigma$-field of $\R^d$), $\P (x_t\in A)=\int_A p(s,t,x,y)\d y$ such that there exists $ C:=C(T,\Theta,\rho) \geq 1$ such that for all  $(x,y)\in \R^d,$
\begin{align}
C^{-1} \bar{p}_\alpha (t-s,y-x) \leq p(s,t,x,y) &\leq C \bar{p}_\alpha (t-s,y-x),\\
\left| \nabla_x p(s,t,x,y) \right| &\leq \frac{C}{(t-s)^{\frac{1}{\alpha}}} \bar{p}_\alpha (t-s,y-x),\\
\forall (y,y')\in \R^d, \qquad \left| p(s,t,x,y) - p(s,t,x,y')\right| &\leq \frac{C |y-y'|^\rho}{(t-s)^{\frac{\rho}{\alpha}}} \left( \bar{p}_\alpha (t-s,y-x) +\bar{p}_\alpha (t-s,y'-x)\right),\\
\forall (y,y')\in \R^d, \qquad \left|\nabla_x p(s,t,x,y) - \nabla_x p(s,t,x,y')\right| &\leq \frac{C |y-y'|^\rho}{(t-s)^{\frac{\rho +1}{\alpha}}} \left( \bar{p}_\alpha (t-s,x-y) +\bar{p}_\alpha (t-s,x-y')\right),
\end{align}
for any $\rho \in (-\beta,\gamma-\beta)$, where $\gamma:= \beta - \frac{1-\alpha + \frac{\alpha}{r} + \frac{d}{p}}{2}$ is the ``gap to singularity".
\end{theorem}

\begin{remark}[Logarithmic gradient estimates.] Note that, in the current strictly stable regime ($\alpha \in (1,2)$) and given the previous theorem, one can easily compute global logarithmic gradient estimates for $p$:
	\begin{align*}
		|\nabla_x \log  p(s,t,x,y)| = \frac{|\nabla_x  p(s,t,x,y)|}{ p(s,t,x,y)}	\leq \frac{C}{(t-s)^{\frac{1}{\alpha}}} .
	\end{align*}
\end{remark}
\noindent The sketch of the proof of Theorem \ref{thm-main} is as follows:
\begin{itemize}
	\item Take a smooth $b^m\in C^{\infty}_b$ to approach $b$ and consider the \textit{mollified} equation 
	\begin{equation}
		\d X_t^m = b^m(t,X_t^m)\d t + \d Z_t.
	\end{equation}
	\item Compute estimates on the density of $(X^m_t)$ which are uniform in $m$, using a Duhamel expansion and a normalization method first introduced by \cite{MPZ21} (Brownian setting with unbounded Hölder drift) and then exploited in \cite{JM21} (Brownian setting with $L^q-L^p$ drift) and \cite{MZ22} (unbounded drift, stable driven with multiplicative isotropic noise).
	\item Conclude with a compactness argument.
\end{itemize}
\subsection{Weak solutions}\label{subsec-weak-sols}

Although mild solutions allow for a formal discussion on the density of the underlying process in the SDE \eqref{sde}, they do not exhibit anything about its dynamics nor about its SDE interpretation. In order to build the dynamics of the equation, \cite{CdRM22} introduced a weak formulation of the problem. To such end, they used the notion of $L^\ell$ stochastic Young integral, in the sense of the definition first introduced by \cite{CG16} and  \cite{DD16}.\\

In order to define the upcoming notion of solution, we need slightly stronger assumptions on $\alpha,\beta$. We say that $\alpha,\beta$ satisfy a \textbf{good relation for the dynamics (GRD)} if the following holds:
\begin{equation}\label{GRD}\tag{GRD}
	\alpha \in \left( \frac{1+\frac{d}{p}}{1-\frac{1}{r}},2 \right)\qquad \beta \in \left( \frac{1-\alpha+\frac{2d}{p}+\frac{2\alpha}{r}}{2} ,0\right).
\end{equation}
This stronger condition is required in order for the following definition to make sense:
\begin{definition}\label{weak-sol}
	We call \textbf{weak} solution of the formal SDE (\ref{sde}) a pair $(Y,Z)$ of adapted processes on a filtered space $(\Omega,\mathcal{F},\{ \mathcal{F}_T \}_{t\geq 0},\mathbb{P})$ such that $Z$ is an $\{ \mathcal{F}_T \}_{t\geq 0}$ $\alpha$-stable process and $(Y,Z)$ satisfies
	\begin{equation}\label{weak-sol-eq}
		Y_t = x +\int_0^t \mathfrak{B}(s,Y_s,\d s) + Z_t, \qquad \mathbb{P}\mathrm{-a.s.}, \qquad \mathbb{E} \left|\int_0^t \mathfrak{B}(s,Y_s,\d s) \right|<\infty
	\end{equation}
	for any $t\in [0,T]$, where
	\begin{equation}\label{dyn-b}
		\mathfrak{B} : (v,x,h) \mapsto  \int_0^h \d r \int_{\R^d} p_\alpha (h-r,x-y) b(v+r,y)\d y
	\end{equation}
	and where the integral in (\ref{weak-sol-eq}) is understood as an $L^1$ stochastic Young integral and imposes the stronger \textbf{(\ref{GRD})} condition.\\
	
\end{definition}

With this explicit definition, it becomes fathomable to develop numerical schemes for the SDE. In particular, as Theorem \ref{thm-main} is proved under \textbf{(\ref{GR})}, it is valid under the stronger \textbf{(\ref{GRD})} conditions and thus holds for the density of weak solutions.

\section{Technical results and Besov spaces}\label{sec-lemmas}
\subsection{A few useful lemmas}
In this section, we introduce a few lemmas that will be used in our proofs, and which we prove in Section \ref{sec-proof}. The following Lemma \ref{deriv-density-lemma} will be profusely used. It indicates how $p_\alpha$ and $\bar{p}_\alpha$ can be bounded, and specifies the relationship between the time and space scales for $\alpha$-stable densities.

\begin{lemma}[Bounds and $L^p$ estimates for the noise density]\label{deriv-density-lemma}\label{pre-big-lemma}\hspace{1cm}
\begin{itemize}
\item Time scales for spatial moments\\
$\forall \alpha \in (1,2),\zeta \in (0,\alpha),\exists C>0 : \forall t>0$, 

\begin{equation}\label{spatial-moments}
\int_{\R^d} \bar{p}_\alpha (t,y) |y|^\zeta \mathrm{d}y \leq C t^{\frac{\zeta}{\alpha}}.
\end{equation}
\item Bounds for space and time derivatives of the $\alpha$-stable kernel.\\
$\forall \ell\in \{1,2\},\forall  t>0,\forall y \in \R^d$,

\begin{equation}\label{space-time-derivatives-bounds}
\left| D_y^\ell p_\alpha (t,y) \right|\lesssim \frac{1}{t^{\frac{\ell}{\alpha}}}\bar{p}_\alpha (t,y),
\end{equation} 
\item Distortion\\
$\forall t \in\R_+, \forall \ell\geq 1$,
\begin{equation}\label{disto}
	\Vert \bar{p}_\alpha (t,\cdot) \Vert_{L^{\ell'}} \lesssim t^{-\frac{d}{\alpha \ell}}.
\end{equation}
\item Convolution\\
$\forall x,y \in \R^d$, $\forall 0 \leq s \leq u \leq t$, $\forall \ell\geq 1$,
\begin{equation}\label{p-q-convo}
	\Vert \bar{p}_\alpha (t-u,\cdot-y) \bar{p}_\alpha (u-s,x-\cdot)\Vert_{L^{\ell'}} \lesssim \left[ \frac{1}{(t-u)^{\frac{d}{\alpha  \ell}}} +\frac{1}{(u-s)^{\frac{d}{\alpha  \ell}}}\right] \bar{p}_\alpha (t-s,x-y).
\end{equation}
\end{itemize}
\end{lemma}

\begin{remark} If $\ell'=1$, (\ref{p-q-convo}) reads
	$$\int \bar{p}_\alpha (t-u,z-y) \bar{p}_\alpha (u-s,x-z) \d z \lesssim \bar{p}_\alpha (t-s,x-y),$$
	and in this case, the proof is simply based on the fact that $\bar{p}_\alpha \asymp p_\alpha$ and the convolution properties of $p_\alpha$. For $\ell'>1$, singularities with the same scale as in \eqref{disto} appear.
\end{remark}

\begin{lemma}[Taylor-Laplace for stable densities]\label{holder-noise}
	
	For all $x,w,z \in \R^d, t > 0$ s.t. $|w-z|\lesssim t^{1/\alpha}$, $\forall \ell\in \{ 0,1\}, \forall \zeta \in (0,1]$,
\begin{itemize}
	\item Bounds for the stable density
	\begin{equation}\label{taylor-stable-density}
		|D^\ell p_\alpha (t,x-w) - D^\ell p_\alpha (t,x-z)| \lesssim \frac{|z-w|^\zeta}{t^{\frac{\ell+\zeta}{\alpha}}}\bar{p}_\alpha (t,x-w).
	\end{equation}
	\item Bounds for $\bar{p}_\alpha$
	\begin{equation}\label{taylor-approx-density}
		|D^\ell \bar{p}_\alpha (t,x-w) - D^\ell \bar{p}_\alpha (t,x-z)| \lesssim \frac{|z-w|^\zeta}{t^{\frac{\ell+\zeta}{\alpha}}}\bar{p}_\alpha (t,x-w).
	\end{equation}
\end{itemize}
\end{lemma}

\begin{remark}
From the definiton of $\bar{p}_\alpha(t,x)=\frac{C_\alpha}{t^{\frac{d}{\alpha}}} \frac{1}{\left( 1+ \frac{|x|}{t^{\frac{1}{\alpha}}} \right)^{d+\alpha}}$, one can gather the following:\\
Let $x\in \R^d$ and $t>0$.\\
	\begin{itemize}
		\item If $|x|\ge t^{\frac{1}{\alpha}}$ (off-diagonal regime),
		\begin{equation}\label{off-diag}
			\bar{p}_\alpha (t,x) \asymp \frac{t}{|x|^{d+\alpha}}.
		\end{equation}
		\item If $|x|\leq t^{\frac{1}{\alpha}}$ (diagonal regime),
		\begin{equation}\label{diag}
			\bar{p}_\alpha (t,x) \asymp \frac{1}{t^{\frac{d}{\alpha}}}.
		\end{equation}
	\end{itemize}
\end{remark}
%

\begin{lemma}[Besov estimates for $\bar{p}_\alpha$]\label{big-lemma}\hspace{0.5cm}
\begin{itemize}
	\item $\forall 0\leq s \leq u \leq t$, $\forall (x,y)\in \R^d$, $\forall \zeta \in (-\beta,1]$, $\forall j,k \in \{ 0,1\}$,
	\begin{align}\label{big-lemma-1}
		\Vert\nabla^j &\bar{p}_{\alpha} (u-s,x-\cdot) \nabla^k p_\alpha (t-u,y-\cdot) \Vert_{\B^{-\beta}_{p',q'}} \nonumber \\
		&\lesssim \frac{\bar{p}_\alpha (t-s,x-y)}{(u-s)^{\frac{j}{\alpha}}(t-u)^{\frac{k}{\alpha}}} (t-s)^{\frac{\beta}{\alpha}}\left[ \frac{1}{(t-u)^{\frac{d }{\alpha  p}}} +\frac{1}{(u-s)^{\frac{d }{\alpha  p}}}\right] \left[(t-s)^{\frac{\zeta}{\alpha}}\left( \frac{1}{(t-u)^{\frac{\zeta }{\alpha}}} + \frac{1}{(u-s)^{\frac{\zeta }{\alpha}}}\right) +1 \right] .
	\end{align}
	\item  $\forall 0\leq s \leq u \leq t$, $\forall (x,y,w)\in \R^d$, $\forall \zeta \in (-\beta,1]$,
	\begin{align}\label{big-lemma-2}
		&\left\Vert \bar{p}_{\alpha} (u-s,x-\cdot)\left[\frac{\nabla p_\alpha (t-u,w-\cdot)}{\bar{p}_{\alpha} (t-s,w-x)} -\frac{\nabla p_\alpha (t-u,y-\cdot)}{\bar{p}_{\alpha} (t-s,y-x)} \right] \right\Vert_{ \B_{p',q'}^{-\beta}} \nonumber \\& \qquad\lesssim \frac{|w-y|^\zeta }{(t-u)^{\frac{\zeta+1}{\alpha}}} (t-s)^{\frac{\beta}{\alpha}} \left[ \frac{1}{(t-u)^{\frac{d }{\alpha  p}}} +\frac{1}{(u-s)^{\frac{d }{\alpha  p}}}\right] \left[(t-s)^{\frac{\zeta}{\alpha}}\left( \frac{1}{(t-u)^{\frac{\zeta }{\alpha}}} + \frac{1}{(u-s)^{\frac{\zeta }{\alpha}}}\right) +1 \right].
	\end{align}
\end{itemize}

\end{lemma}

Note that, in \eqref{big-lemma-1} and \eqref{big-lemma-2}, the additional term $(t-s)^{\frac{\zeta}{\alpha}}\left( \frac{1}{(t-u)^{\frac{\zeta }{\alpha}}} + \frac{1}{(u-s)^{\frac{\zeta }{\alpha}}}\right)$ will disappear though time integration and is mainly due to technical considerations related to the thermic characterizations of the Besov spaces considered for the above norms.\\

\noindent For the rest of this paper, we will use the notation 
\begin{equation}\label{lust}
	\mathfrak{L}(u,s,t,\zeta) :=  \frac{(t-s)^{\frac{\beta}{\alpha}}}{(t-u)^{\frac{1}{\alpha}}} \left[ \frac{1}{(t-u)^{\frac{d }{\alpha  p}}} +\frac{1}{(u-s)^{\frac{d }{\alpha  p}}}\right] \left[(t-s)^{\frac{\zeta}{\alpha}}\left( \frac{1}{(t-u)^{\frac{\zeta }{\alpha}}} + \frac{1}{(u-s)^{\frac{\zeta }{\alpha}}}\right) +1 \right],
\end{equation}
so that, for $(j,k)=(0,1)$, (\ref{big-lemma-1}) reads $$\Vert \bar{p}_{\alpha} (u-s,x-\cdot) \nabla p_\alpha (t-u,y-\cdot) \Vert_{\B^{-\beta}_{p',q'}} \lesssim \bar{p}_\alpha (t-s,x-y) \mathfrak{L}(u,s,t,\zeta).$$

\subsection{Results in Besov spaces}\label{subsec-besov}

\begin{paragraph}{Thermic characterization of Besov spaces\\[0.5cm]}
From the work of Triebel \cite{Tri88}, we can use the following thermic characterization of Besov spaces (see also \cite{Stein70}, Chapter 5, paragraph 4):

\begin{propo}For $\vartheta \in \R, m\in (0,+\infty], \ell \in (0,\infty], \B_{\ell,m}^\vartheta (\R^d)= \left\{ f \in \mathcal{S}'(\R^d) : \Vert f\Vert_{\mathcal{H}_{\ell,m}^\vartheta,\tilde{\alpha}} < \infty\right\}$, where: 
\begin{align}\label{thermic-char}
 \Vert f\Vert_{\mathcal{H}_{\ell,m}^{\vartheta},\tilde{\alpha}} &:= \Vert \phi(D) f \Vert_{L^\ell} + \left\{ \begin{aligned} &\left( \int_0^1 \frac{\d v}{v} v^{(n-\frac{\vartheta}{\tilde{\alpha}})m} \Vert \partial_v^n p_{\tilde{\alpha}} (v,\cdot) \star f \Vert_{L^\ell}^m\right)^{\frac{1}{m}}, \qquad &m<\infty,\\
&\sup_{v\in (0,1]} v^{n-\frac{\vartheta}{\tilde{\alpha}}} \Vert \partial_v^n p_{\tilde{\alpha}}(v,\cdot) \star f \Vert_{L^\ell}, &m=\infty,
\end{aligned}  \right. \nonumber \\
 &=: \Vert \phi(D) f \Vert_{L^\ell} + \mathcal{T}_{\ell,m}^{\vartheta} [f].
\end{align}
\end{propo}

We call $\mathcal{T}_{\ell,m}^{\vartheta} [f]$ the thermic part due to it involving a convolution with a heat kernel $p_{\tilde{\alpha}}$. By default, $\Vert \phi(D) f \Vert_{L^\ell}$ will thus be denoted by ``non-thermic part". Here, the choice of $\phi$, $\tilde{\alpha}$ and $n$ are free so long as they satisfy $n-\frac{\vartheta}{\tilde{\alpha}}>0$.\\

\textbf{Choice of $\tilde{\alpha}$:} Triebel suggested using this characterization with Poisson ($\tilde{\alpha} =1$) kernel or Gaussian ($\tilde{\alpha}=2$) kernel, but \cite{Tri88} also includes any $\tilde{\alpha} \in (1,2)$. In the following, we will use this with $\tilde{\alpha}= \alpha$. This characterization is very convenient when working with Besov norms of stable densities (or, for that matter, the approximations $\bar{p}_\alpha$ of such densities) because of the convolution properties they enjoy.\\

\textbf{Choice of $\phi$:} as mentioned, since we will be applying this characterization with densities, it would be wise to choose $\phi$ so that the $L^\ell$ part is not too dissimilar to the $L^\ell$ norm of the thermic part. In fact, the choice of $\phi$ is not so critical, because the most delicate part is the thermic part, where time singularities appear. We will work with $\phi \in \mathcal{S}$ s.t. $\phi(0)\neq 0$.\\

\textbf{Choice of $n$:} in our case, $\vartheta / \alpha <1$, so any integer $n$ will work. We will see that in our computations, taking $n=1$ allows to use cancellation arguments and is sufficient to ensure the thermic part in (\ref{thermic-char}) is a convergent integral.
\end{paragraph}

\begin{paragraph}{Useful inequalities in Besov spaces\\}
\begin{itemize}
\item From \cite{LR02}, Proposition 3.6, we have the following duality inequality:\\
$\forall m,\ell,\vartheta$, with $m'$ and $\ell'$ respective conjugates of $m$ and $\ell$, and $(f,g)\in \B_{\ell,m}^\vartheta \times \B_{\ell',m'}^{-\vartheta}$,

\begin{equation}\label{dual-ineq}
\left| \int f(y)g(y) \mathrm{d}y \right| \leq \Vert f \Vert_{\B_{\ell,m}^\vartheta} \Vert g \Vert_{\B_{\ell',m'}^{-\vartheta}}.
\end{equation}

\item From \cite{Sawano18}, Theorem 4.37, we have the following Hölder inequality (product rule):\\
$\forall p,q,s$ and $\forall \rho > \max \left\{ s,d\left(\frac{1}{p}-1 \right)_+ - s\right\}$, $\forall (f,g)\in \B_{\infty,\infty}^\rho \times \B_{p,q}^s$,
\begin{equation}\label{holder-besov}
\Vert f\cdot g\Vert_{ B_{p,q}^s} \lesssim \Vert f \Vert_{\B_{\infty,\infty}^\rho} \Vert g \Vert_{ \B_{p,q}^s}.
\end{equation}
In our setting, as $p\geq 1$ and $\beta<0$, the above condition on $\rho$ reads $\rho>-\beta$.
\end{itemize}
\end{paragraph}
As explained before, our approach consists in approximating the drift with a sequence of smooth functions on which to perform computations. In the next proposition, we explicit how such approximation works: the sole sensitive case is when $r=+\infty$. Namely, we can state the following, which is proved in \cite{CdRJM22} (see also \cite{IJ20} for drifts in $L^\infty ([0,T],H_q^{\beta})$) :
\begin{propo}\label{approx-lemma}[Smooth approximation of the drift] Let $b\in L^r -\B^\beta_{p,q}$ with $\beta \in (-1,0]$, $1\leq p,q\leq \infty$. There exists a time-space sequence of smooth bounded functions $(b^m)_{m\in \N}$ such that
	$$\Vert b-b^m\Vert_{L^{\tilde{r}} -\B^{\tilde{\beta}}_{p,q}} \underset{m \rightarrow \infty}{\longrightarrow} 0,\qquad \forall \tilde{\beta} < \beta,$$
	with $\tilde{r} = r$ if $r<\infty$ and for any $\tilde{r}<\infty$ otherwise. Moreover, there exists $\kappa \geq 1:$
	\begin{equation}\label{maj-bm}
		\sup_{m\in \N} \Vert b^m \Vert_{L^{\tilde{r}} -\B^\beta_{p,q}} \leq \kappa \Vert b \Vert_{L^{\tilde{r}} -\B^\beta_{p,q}}.
	\end{equation}
	
\end{propo}

Note that this approximation induces a slight loss in space regularity. For this paper, this is of no concern as Theorem \ref{thm-main} does not hold for $\rho = \gamma-\beta$.

\section{Estimates on the mollified SDE}\label{sec-estimates}
In this section, we only consider the \textit{mollified} SDE

\begin{equation}\label{mol-sde}
	\d X_t^m = b^m (t,X_t^m) \d t + \d Z_t,
\end{equation}
where $(b^m)_{m\in \N} \in C_b^\infty$ is an approximating sequence of the drift, as given by Proposition \ref{approx-lemma}. As thus, this SDE is a classical one, and we have strong well-posedness and uniqueness. In this setting, it is known that the density of $(X_t^m)_{t\geq s}$ exists for $t>s$ (see e.g. \cite{Kol00} or \cite{Le85} for a more general additive noise). We will prove the following theorem:

\begin{theorem}\label{mol-thm} Fix the parameters $T>0$ and $\Theta = \{\alpha,d,\beta,r,p,q,\Vert b \Vert_{L^{r} -\B^\beta_{p,q}}\}$. Assume  \textbf{(\ref{GR})} holds. For any $m$, consider the solution $\mathbb{P}^m$ to the martingale problem associated with $(b^m,\mathcal{L}^\alpha,x)$ starting at time $s$ and denote $(x_t^m)_{t\in {[s,T]}}$ the associated canonical process. For all $t\in (s,T]$, $x_t^m$ admits a density $p^m(s,t,x,\cdot)$ such that there exists $C:=C(T,\Theta,\rho) \geq 1$ such that for all $(x,y)\in \R^d,$
	\begin{align}
		C^{-1} \bar{p}_\alpha (t-s,y-x) \leq p^m(s,t,x,y) &\leq C \bar{p}_\alpha (t-s,y-x), \label{mol-thm-1}\\
		\left| \nabla_x p^m(s,t,x,y) \right| &\leq \frac{C}{(t-s)^{\frac{1}{\alpha}}} \bar{p}_\alpha (t-s,y-x), \label{mol-thm-2}\\
		\forall (y,y')\in \R^d, \qquad \left| p^m(s,t,x,y) - p^m(s,t,x,y')\right| &\leq \frac{C |y-y'|^\rho}{(t-s)^{\frac{\rho}{\alpha}}} \left( \bar{p}_\alpha (t-s,y-x) +\bar{p}_\alpha (t-s,y'-x)\right), \label{mol-thm-3}\\
		\forall (y,y')\in \R^d, \qquad \left|\nabla_x p^m(s,t,x,y) - \nabla_x p^m(s,t,x,y')\right| &\leq \frac{C |y-y'|^\rho}{(t-s)^{\frac{\rho +1}{\alpha}}} \left( \bar{p}_\alpha (t-s,x-y) +\bar{p}_\alpha (t-s,x-y')\right), \label{mol-thm-4}
		\end{align}
	for any $\rho \in (-\beta,\gamma-\beta)$, where $\gamma:= \beta - \frac{1-\alpha + \frac{\alpha}{r} + \frac{d}{p}}{2}$ is the ``gap to singularity".\\
	Note that, in the current mollified setting, $X^m_t=x_t^m$ is the strong solution to \eqref{mol-sde}.
\end{theorem}

\begin{remark} We insist that, in the context of Proposition \ref{approx-lemma}, this statement is uniform in $m$ as $C$ does not depend on $m$. We will see in the proof that this is made possible by \eqref{maj-bm}. We could in fact already obtain those bounds from \cite{Kol00}, but they would not be uniform in $m$.
	
\end{remark}
\begin{remark}\label{beta-tilde} As previously mentioned, when using Proposition \ref{approx-lemma}, the approximating sequence does not exactly have the same regularity and integrability indexes as the original drift. However, we chose to state Theorem \ref{mol-thm} with the original $r$ and $\beta$ because, as we do not obtain large inequalities but strict ones ($\rho \in (-\beta,\gamma - \beta)$), the final result remains the same. 
	
\end{remark}

\begin{proof} We will prove Theorem \ref{mol-thm} for $T\in (0,1)$. To extend this proof to any $T>0$, it suffices to use the Chapman-Kolmogorov property of $p_\alpha$.\\

As equation (\ref{mol-sde}) can be understood in a classical way, we can perform a Duhamel expansion on the density of the solution (see e.g. \cite{MZ22}). Namely, $\forall 0\leq s < t \leq T, \forall (x,y) \in \R^d$,
\begin{equation}\label{duhamel}
p^m(s,t,x,y) = p_\alpha (t-s,y-x)+\int_s^t \int p^m (s,u,x,z) b^m (u,z) \nabla_z p_\alpha (t-u,y-z) \d z \d u.
\end{equation}
Let us now denote, for fixed $(s,x)\in [0,1]\times \R^d$, $$h_{s,x}^m(t,y) :=  \frac{p^m (s,t,x,y)}{\bar{p}_{\alpha} (t-s,y-x)}.$$
\textbf{For the sake of clarity, until the end of this section we will omit the $m$ in the previous $h_{s,x}^m$, and denote $h_{s,x}=h_{s,x}^m$.\\}

We already know that $ p_\alpha \asymp  \bar{p}_\alpha$, hence we can write:
\begin{align*}
h_{s,x}(t,y) & \leq C + \frac{1}{\bar{p}_{\alpha} (t-s,y-x)} \left| \int_s^t \int \frac{p^m (s,u,x,z)}{\bar{p}_{\alpha} (u-s,z-x)} b^m (u,z) \bar{p}_{\alpha} (u-s,z-x) \nabla p_\alpha (t-u,y-z) \d z \d u \right|\\
& \lesssim 1 + \frac{1}{\bar{p}_{\alpha} (s,t,x,y)}  \int_s^t\left| \int h_{s,x}(u,z) b^m (u,z) \bar{p}_{\alpha} (s,u,x,z) \nabla p_\alpha (t-u,y-z) \d z \right| \d u .
\end{align*}
From this point, our goal is to use a Gronwall-Volterra lemma on this expansion. This will give us bounds on $h$, which we need to be uniform in $m$. In our case, we do not know much about $b^m$, and the most we might be able to rely on is that $\Vert b^m - b \Vert_{L^r-\B^{\beta}_{p,q}} \rightarrow 0$. On the flipside, we know a lot about $\bar{p}_\alpha$, $p_\alpha$ and their derivatives. In particular, it is very smooth, and we should be able to control its Besov norm rather well. Hence we will use the duality inequality (\ref{dual-ineq}) to derive:
\begin{align*}
h&_{s,x}(t,y) \lesssim 1 +\frac{1}{\bar{p}_{\alpha} (t-s,y-x)}\int_s^t   \Vert h_{s,x}(u,\cdot)  b^m (u,\cdot) \Vert_{\B^{\beta}_{p,q}} \Vert \bar{p}_{\alpha} (u-s,\cdot-x) \nabla p_\alpha (t-u,y-\cdot) \Vert_{\B^{-\beta}_{p',q'}}  \d u\\
& \lesssim  1 +\frac{1}{\bar{p}_{\alpha} (t-s,y-x)}\int_s^t  \Vert h_{s,x}(u,\cdot) \Vert_{\B^{\rho}_{\infty,\infty}} \Vert b^m (u,\cdot) \Vert_{\B^{\beta}_{p,q}} \Vert \bar{p}_{\alpha} (u-s,\cdot-x) \nabla p_\alpha (t-u,y-\cdot) \Vert_{\B^{-\beta}_{p',q'}} \d u,
\end{align*}
where the last inequality was obtained using (\ref{holder-besov}), with any $\rho > \max \left\{ \beta,d \left( \frac{1}{p}-1\right)_+ - \beta \right\} =  - \beta$. \\
Using Lemma \ref{big-lemma}, we get
\begin{align}\label{l-inf-bound}
h_{s,x}(t,y) \lesssim  1 +\int_s^t   \Vert h_{s,x}(u,\cdot) \Vert_{\B^{\rho}_{\infty,\infty}} \Vert b^m (u,\cdot) \Vert_{\B^{\beta}_{p,q}}  \mathfrak{L}(u,s,t,\rho)  \d u.
\end{align}
We now need to retrieve $\Vert h_{s,x}(u,\cdot) \Vert_{\B^{\rho}_{\infty,\infty}}$ on the l.h.s. to use a Gronwall-Volterra lemma.
\begin{align*}
\Vert h_{s,x}(t,\cdot) \Vert_{\B^{\rho}_{\infty,\infty}} &= \Vert h_{s,x}(t,\cdot) \Vert_{L^{\infty}} + \sup_{v\in (0,1]} v^{1-\frac{\rho}{\alpha}} \Vert \partial_v p_\alpha (v,\cdot) \star h_{s,x}(t,\cdot)\Vert_{L^\infty}\\
&= \Vert h_{s,x}(t,\cdot) \Vert_{L^{\infty}} + \mathcal{T}^\rho_{\infty,\infty}[h_{s,x}(t,\cdot)].
\end{align*}

The non-thermic part can already be estimated from (\ref{l-inf-bound}). For the thermic part, we introduce the following technical lemma, whose proof is postponed to Section \ref{sec-proof}:

\begin{lemma}[Thermic part of $\Vert h_{s,x}(u,\cdot) \Vert_{\B^{\rho}_{\infty,\infty}}$]\label{thermic-h} $\forall 0\leq s \leq t,\forall x\in \R^d,\forall \rho \in (-\beta,1),$
	\begin{equation*}
		\mathcal{T}_{\infty,\infty}^\rho [h_{s,x}(t,\cdot)] \lesssim \frac{1}{(t-s)^{\frac{\rho}{\alpha}}} \left(1 + \int_s^t  \Vert h_{s,x}(u,\cdot) \Vert_{\B_{\infty,\infty}^\rho} \Vert b^m(u,\cdot) \Vert_{\B_{p,q}^\beta}  \mathfrak{L}(u,s,t,\rho) \left[\frac{(t-s)^{\frac{\rho}{\alpha}}}{(t-u)^{\frac{\rho}{\alpha}}} +1 \right]  \d u\right).
	\end{equation*}
\end{lemma}
This lemma indicates that the thermic part of $\Vert h_{s,x}(u,\cdot) \Vert_{\B^{\rho}_{\infty,\infty}}$ is not homogeneous to its non-thermic part. We therefore introduce a normalized version of $\Vert h_{s,x}(u,\cdot) \Vert_{\B^{\rho}_{\infty,\infty}}$ on which to perform a Gronwall-Volterra lemma, accounting for the right time singularity. Let
\begin{align}
	g(s,x,t):=g^m(s,x,t):=&\Vert h_{s,x}^m(t,\cdot) \Vert_{L^{\infty}} + (t-s)^{\frac{\rho}{\alpha}} \mathcal{T}^\rho_{\infty,\infty}[h_{s,x}^m(t,\cdot)] \nonumber\\
	=& \Vert h_{s,x}(t,\cdot) \Vert_{L^{\infty}} + (t-s)^{\frac{\rho}{\alpha}} \mathcal{T}^\rho_{\infty,\infty}[h_{s,x}(t,\cdot)].
\end{align}
With the previous lemma, we can write
\begin{align}\label{maj-g1-1}
	g(s,x,t) &\lesssim 1 + \int_s^t  \Vert h_{s,x}(u,\cdot) \Vert_{\B_{\infty,\infty}^\rho} \Vert b^m(u,\cdot) \Vert_{\B_{p,q}^\beta}  \mathfrak{L}(u,s,t,\rho) \left[\frac{(t-s)^{\frac{\rho}{\alpha}}}{(t-u)^{\frac{\rho}{\alpha}}} +1 \right] \d u .
\end{align}
Notice that, because $(u-s)\leq(t-s)\leq 1$,
\begin{align*}
	\Vert (u-s)^{\frac{\rho}{\alpha}} h_{s,x}(u,\cdot) \Vert_{\B^{\rho}_{\infty,\infty}} = \Vert (u-s)^{\frac{\rho}{\alpha}} h_{s,x}(u,\cdot) \Vert_{L^{\infty}} +(u-s)^{\frac{\rho}{\alpha}} \mathcal{T}^\rho_{\infty,\infty}[h_{s,x}(u,\cdot)] \leq g(s,x,u).
\end{align*}
Because of this, (\ref{maj-g1-1}) yields
\begin{align}
	g(s,x,t) &\lesssim 1 + \int_s^t \frac{g(s,x,u)}{(u-s)^{\frac{\rho}{\alpha}}} \Vert b^m(u,\cdot) \Vert_{\B_{p,q}^\beta}  \mathfrak{L}(u,s,t,\rho) \left[\frac{(t-s)^{\frac{\rho}{\alpha}}}{(t-u)^{\frac{\rho}{\alpha}}} +1 \right] \d u \nonumber.
\end{align}
We now apply a Gronwall-Volterra lemma:
\begin{align}\label{g-bound}
g(s,x,t)  \lesssim  1 +\int_s^t & \frac{\Vert b^m(u,\cdot) \Vert_{\B_{p,q}^\beta}}{(u-s)^{\frac{\rho}{\alpha}}}   \mathfrak{L}(u,s,t,\rho) \left[\frac{(t-s)^{\frac{\rho}{\alpha}}}{(t-u)^{\frac{\rho}{\alpha}}} +1 \right] \nonumber\\&  \exp\left[ \int_u^t \frac{\Vert b^m(v,\cdot) \Vert_{\B_{p,q}^\beta}}{(v-s)^{\frac{\rho}{\alpha}}}   \mathfrak{L}(v,s,t,\rho) \left[\frac{(t-s)^{\frac{\rho}{\alpha}}}{(t-v)^{\frac{\rho}{\alpha}}} +1 \right]  \d v \right] \d u.
\end{align}
Let us first focus on the integral in $v$. We first treat $\Vert b^m(v,\cdot) \Vert_{\B_{p,q}^\beta}$ using Hölder's inequality:
\begin{align*}
	\int_u^t \frac{\Vert b^m(v,\cdot) \Vert_{\B_{p,q}^\beta}}{(v-s)^{\frac{\rho}{\alpha}}} &  \mathfrak{L}(v,s,t,\rho) \left[\frac{(t-s)^{\frac{\rho}{\alpha}}}{(t-v)^{\frac{\rho}{\alpha}}} +1 \right]  \d v \\ &\leq \Vert b^m \Vert_{L^{r} -\B^{\beta}_{p,q}} \left( \int_u^t \frac{ \mathfrak{L}(v,s,t,\rho)^{r'}}{(v-s)^{\frac{\rho {r'}}{\alpha}}}   \left[\frac{(t-s)^{\frac{\rho}{\alpha}}}{(t-v)^{\frac{\rho}{\alpha}}} +1 \right]^{r'}  \d v\right)^{\frac{1}{r'}}.
\end{align*}
Here, we want to specify the conditions under which this integral converges. Let us first make it explicit:
\begin{align*}
		\int_u^t \frac{ \mathfrak{L}(v,s,t,\rho)^{r'}}{(v-s)^{\frac{\rho {r'}}{\alpha}}}   \left[\frac{(t-s)^{\frac{\rho}{\alpha}}}{(t-v)^{\frac{\rho}{\alpha}}} +1 \right]^{r'}  \d v \nonumber &= \int_u^t \frac{ 1}{(v-s)^{\frac{\rho {r'}}{\alpha}}} \frac{(t-s)^{\frac{r' \beta}{\alpha}}}{(t-v)^{\frac{r'}{\alpha}}} \left[\frac{1}{(t-v)^{\frac{d}{\alpha p}}} + \frac{1}{(v-s)^{\frac{d}{\alpha p}}} \right]^{r'} \\
	& \qquad \times     \left[ \frac{(t-s)^{\frac{\rho}{\alpha}}}{(t-v)^{\frac{\rho}{\alpha}}} +\frac{(t-s)^{\frac{\rho}{\alpha}}}{(v-s)^{\frac{\rho}{\alpha}}}+1\right]^{r'}\left[\frac{(t-s)^{\frac{\rho}{\alpha}}}{(t-v)^{\frac{\rho}{\alpha}}} +1\right]^{r'}  \d v .
\end{align*}
As $0\leq s \leq u \leq v\leq t$, singularities appear for $v\rightarrow s$ and $v\rightarrow t$, and it is sufficient to prove that the following integral is convergent:
\begin{equation}\label{integ-expo}
	\int_s^t \frac{ \mathfrak{L}(v,s,t,\rho)^{r'}}{(v-s)^{\frac{\rho {r'}}{\alpha}}}   \left[\frac{(t-s)^{\frac{\rho}{\alpha}}}{(t-v)^{\frac{\rho}{\alpha}}} +1 \right]^{r'}  \d v .
\end{equation}
Setting $v=(t-s)\lambda +s$ and accounting only for the most singular terms in (\ref{integ-expo}), we get
\begin{align}\label{beta-chgt-var}
	\int_s^t \frac{ \mathfrak{L}(v,s,t,\rho)^{r'}}{(v-s)^{\frac{\rho {r'}}{\alpha}}} &  \left[\frac{(t-s)^{\frac{\rho}{\alpha}}}{(t-v)^{\frac{\rho}{\alpha}}} +1 \right]^{r'}  \d v \nonumber \\&\lesssim (t-s)^{\frac{r'}{\alpha }(\beta - \rho - 1 - \frac{d}{p})+1}\int_0^1  \frac{1}{\lambda^{\frac{r'}{\alpha}(2\rho + \frac{d}{p})}(1-\lambda)^{\frac{r'}{\alpha }}}+ \frac{1}{\lambda^{\frac{\rho r'}{\alpha}}(1-\lambda)^{\frac{r'}{\alpha }(1+\frac{d}{p}+2\rho)}}\d \lambda.
\end{align}
This integral converges if and only if $\frac{r'}{\alpha}(1+\frac{d}{p}+2\rho)<1$. Denote $\eps = \rho + \beta > 0$. Then, 
\begin{equation}\label{serrin}
	\frac{r'}{\alpha} \left(1+\frac{d}{p}+2\rho\right)<1\ssi \eps < \beta - \frac{1-\alpha + \frac{\alpha}{r}+\frac{d}{p}}{2} = \gamma,
\end{equation}
and as we work under \textbf{(\ref{GR})}, the r.h.s. $\gamma$ of (\ref{serrin}) is positive.\\
Again, under  \textbf{(\ref{GR})}, the exponent of $(t-s)^{\frac{r'}{\alpha }(\beta - \rho - 1 - \frac{d}{p})+1}$ in \eqref{beta-chgt-var} is greater than $\gamma$, and in particular, it is positive, meaning there are no singularities in $(t-s)$ in (\ref{beta-chgt-var}).\\
 Hence $\forall \eps \in (0,\gamma)$ (i.e. $\forall \rho \in (-\beta,\gamma-\beta)$),
$$\int_u^t \frac{\Vert b^m(v,\cdot) \Vert_{\B_{p,q}^\beta}}{(v-s)^{\frac{\rho}{\alpha}}}   \mathfrak{L}(v,s,t,\rho) \left[\frac{(t-s)^{\frac{\rho}{\alpha}}}{(t-v)^{\frac{\rho}{\alpha}}} +1 \right]  \d v \lesssim \Vert b^m \Vert_{L^{r} -\B^{\beta}_{p,q}} .$$

\begin{remark} We see here that the threshold $\rho<\gamma-\beta$ is due to integrability of \eqref{beta-chgt-var}, while the constraint $\rho >-\beta$ comes from the above use of a duality inequality.
\end{remark}

Notice that, in (\ref{g-bound}), the same computations and conditions yield that $g(s,x,t)$ is bounded by a constant $C$ which depends only on $T$ and $ \Vert b^m \Vert_{L^{r} -\B^{\beta}_{p,q}}$ in a non-decreasing way. Using \eqref{maj-bm}, we obtain the uniform boundedness of $g$, which in turn yields $\mathcal{T}_{\infty,\infty}^\rho [h_{s,x}(t,\cdot)]\lesssim \frac{1}{(t-s)^{\frac{\rho}{\alpha}}}$. From the definition of $h_{s,x}(t,\cdot)$ and \eqref{l-inf-bound}, we obtain the upper bound of (\ref{mol-thm-1}) and (\ref{mol-thm-3}). To obtain the lower bound of (\ref{mol-thm-1}), it suffices to write:
\begin{align}\label{lower-bound}
	h_{s,x}(t,y) & \geq C - \frac{1}{\bar{p}_{\alpha} (t-s,y-x)} \left| \int_s^t \int \frac{p^m (u-s,z-x)}{\bar{p}_{\alpha} (u-s,z-x)} b^m (u,z) \bar{p}_{\alpha} (u-s,z-x) \nabla p_\alpha (t-u,y-z) \d z \d u \right| \nonumber \\
	&\gtrsim 1 - \int_s^t  \Vert h_{s,x}(u,\cdot) \Vert_{\B_{\infty,\infty}^\rho} \Vert b^m(u,\cdot) \Vert_{\B_{p,q}^\beta}  \mathfrak{L}(u,s,t,\rho) \left[\frac{(t-s)^{\frac{\rho}{\alpha}}}{(t-u)^{\frac{\rho}{\alpha}}} +1 \right] \d u.
\end{align}
And the result follows from the control which we have already performed on $\Vert h_{s,x}(u,\cdot) \Vert_{\B_{\infty,\infty}^\rho}$ under \textbf{(\ref{GR})}: namely, the integral on the r.h.s. of \eqref{lower-bound} is a $o((t-s)^{\rho/\alpha})$.\\
For items (\ref{mol-thm-2}) and (\ref{mol-thm-4}), it suffices to notice that the whole proof remains the same if we add a derivative w.r.t. the initial value $x$, and using (\ref{space-time-derivatives-bounds}) to account for the gradient at the end.  Namely, we would get the Duhamel expansion
\begin{equation*}
	\nabla_x p^m(s,t,x,y) = \nabla_x p_\alpha (t-s,y-x)+\int_s^t \int \nabla_x p^m (s,u,x,z) b^m (u,z) \nabla_z p_\alpha (t-u,y-z) \d z \d u.
\end{equation*}
In turn, this means we have to study $$H_{s,x}^m(t,y) := (t-s)^{\frac{1}{\alpha}} \frac{\nabla_x p^m (s,t,x,y)}{ \bar{p}_{\alpha} (t-s,y-x)}.$$
Computations then remain the same as in this section, up to a factor $\frac{(t-s)^{1/\alpha}}{(u-s)^{1/\alpha}}$ that will disappear through time integration when using the Gronwall-Volterra lemma. To be precise, it exactly adds $r'/\alpha$ to the exponent of $\lambda$ in \eqref{beta-chgt-var}. Importantly, the condition \eqref{serrin} allowing the integral \eqref{beta-chgt-var} to converge remains the same.
Denoting $G^m(s,x,t):=\Vert H_{s,x}^m(t,\cdot) \Vert_{L^{\infty}} + (t-s)^{\frac{\rho}{\alpha}} \mathcal{T}^\rho_{\infty,\infty}[H_{s,x}^m(t,\cdot)]$, this means we obtain the boundedness of $G^m(s,x,t)$, hence (\ref{mol-thm-2}) and, in turn, (\ref{mol-thm-4}).

\end{proof}

\section{From the smooth approximation to the actual SDE}\label{sec-approx}

By Proposition \ref{approx-lemma}, let $(b^m)_{m\in \N}$ be a sequence of smooth bounded functions s.t.
$$\Vert b-b^m\Vert_{L^{\tilde{r}} -\B^{\tilde{\beta}}_{p,q}} \underset{m \rightarrow \infty}{\longrightarrow} 0,\qquad \forall \tilde{\beta} < \beta,$$
with $\tilde{r} = r$ if $r<\infty$ and for any $\tilde{r}<\infty$ otherwise and let $ \kappa \geq 1:$
$$ \sup_{m\in \N} \Vert b^m \Vert_{L^{\tilde{r}} -\B^\beta_{p,q}} \leq \kappa \Vert b \Vert_{L^{\tilde{r}} -\B^\beta_{p,q}}.$$

The following was already discussed in \cite{CdRM22}, but we reproduce it here for the sake of completeness. 

\begin{paragraph}{Tightness of the sequence of probability measures $(\P^m)_{m\in \N}$\\[0.5cm]}
	
Notice that when considering the \textit{mollified} equation (\ref{mol-sde}), for every $m$, the martingale problem associated with $(b^m,\mathcal{L}^\alpha,x)$ is well posed (see \cite{CdRM22}). Let us denote by $\P^m$ its solution and by $(x_t^m)_{t\geq 0}$ the associated canonical process. Let $u_m=(u_m^1,...,u_m^d)$ where, $\forall i, u^i_m$ is a mild solution of the classical Cauchy problem $\mathcal{C}(b^m,\mathcal{L}^\alpha,-b^m_i,0,T)$ (i.e. with terminal condition $u_m^i (T,\cdot)=0$ and source term $-b^m_i$, the $i^{\mathrm{th}}$ component of $b^m$), so that
$$\left( u_m(t,x_t^m) + \int_0^t b^m(s,x_s^m) \d s - u(0,x)\right)_{0\leq t \leq T}$$ is a $\P^m$-martingale, which we can express, through Itô's formula, as
\begin{equation}\label{zvonkin-martingale}
	M_{v,s}(u_m,x^m) :=\int_v^s \int_{\R^d \backslash \left\{ 0 \right\} } \left[u_m(r,x_{r^-}^m + x) - u_m(r,x_{r^-}^m )\right] \tilde{N}(\d r, \d x), \qquad \forall s\geq v,
\end{equation}
	where $\tilde{N}$ is the compensated Poisson measure. 
Itô's formula now writes
\begin{equation}
 x_s^m -x_v^m = 	M_{v,s}(u_m,x^m) +Z_s-Z_v - \left[u_m(s,x_{s}^m) - u_m(v,x_{v}^m )\right] .
\end{equation}

We will use an Aldous criterion to prove the tightness of $(\P^m)_{m\in \N}$, which means we need a control of the form $\mathbb{E}[|X_s^m -X_v^m|^p]\leq c(s-v)^\zeta$ for some $p>0$ and some $\zeta>0$ (see Proposition 34.9 from \cite{Bas11}). Since $\forall i, u_m^i$ is the mild solution of the Cauchy problem $\mathcal{C}(b^m,\mathcal{L}^\alpha,-b_m^i,0,T)$, we can write
\begin{align}
	|u_m(v,x_{v}^m) - u_m(s,x_{s}^m )| &\leq |u_m(v,x_{v}^m) - u_m(v,x_{s}^m)| + |u_m(v,x_{s}^m )- u_m(s,x_{s}^m )|,
\end{align}
and use Proposition 9 from \cite{CdRM22} to get the required space and time controls. Namely, for the spatial part, $\exists C_T$ s.t. $C_T\rightarrow0$ as $T\rightarrow0$ and $|u_m(v,x_{v}^m) - u_m(v,x_{s}^m)|<C_T |x_v^m - x_s^m|$. For the time part, we use the Hölder continuity in time of $u_m$.
For $M_{v,s}(u_m,x^m)$, the control follows from the Burkholder--Davis--Gundy inequality and, finally, for $Z_s-Z_v$, it follows from \eqref{spatial-moments} and the stationarity of $Z$.
\end{paragraph}

\begin{paragraph}{Limit probability measure\\[0.5cm]}
We will now prove that any limit probability measure $\mathbb{ P}$ is a martingale solution to \eqref{sde} in the sense of Definition \ref{mart-sol}. Let $f \in \mathcal{ C}([0,T], \mathcal S(\R^d,\R))$,  $g \in \mathcal{C}^1(\R^d,\R)$ with $Dg \in \B_{\infty,\infty}^{\theta-1}(\R^d,\R^d)$.  Let $u_m\in \mathcal{C}^{0,1}([0,T]\times \R^d)$ be the classical solution of the \textit{mollified} Cauchy problem $\mathcal{C}(b^m,\mathcal{L}^\alpha,f,g,T)$, with $Du_m \in \mathcal{ C}_b^0 ([0,T],\B^{\theta-1-\eps}_{\infty,\infty})$ for some $0<\eps\ll 1 $. 
By Theorem \ref{mol-thm}, we have a uniform control of the modulus of continuity of $u_m$ and $Du_m$. By the Arzelà-Ascoli Theorem, we can extract a subsequence $(u_{m_k},Du_{m_k})_k$ s.t. $(u_{m_k})_k$ and $(Du_{m_k})_k$ converge uniformly on every compact subsets of $[0,T]\times \R^d$ to some functions $u\in \mathcal{C}^{0,1}([0,T]\times \R^d)$ and $Du\in \mathcal{ C}_b^0 ([0,T],\B^{\theta-1-\epsilon}_{\infty,\infty}), \forall \epsilon \in (0,\eps)$ respectively ($Du$ being the space-derivative of $u$). Because of this uniform convergence, (\ref{mild-sol}) holds for the limit, i.e.
\begin{equation}
	\forall (t,x) \in [0,T]\times \R^d, u(t,x)=P^\alpha_{T-t}[g](x) - \int_t^T P_{s-t}^\alpha [f-b\cdot Du ](s,x) \d s,
\end{equation}
hence $u$ is a mild solution to $\mathcal{C}(b,\mathcal{L}^\alpha,f,g,T)$ (see again Remark \ref{beta-tilde} for the handling of the space regularity). Together with a control of the moments of $X^m$ (which we already obtained in the last paragraph), we deduce that 
$$\left( u(t,x_t) + \int_0^t f(s,x_s) \d s - u(0,x_0)\right)_{0\leq t \leq T}$$ 
is a $\P$-martingale.

\end{paragraph}	

\begin{paragraph}{Uniqueness of the limit probability measure\\[0.5cm]}
Let $\P$ and $\tilde{\P}$ be two solution of the martingale problem associated with $(b,\mathcal{L}^\alpha,x_0)$ for some $x_0\in \R^d$. Thus, $\forall f\in \mathcal{ C}([0,T], \mathcal S(\R^d,\R))$, taking $g=0$,
$$u(0,x_0) = \mathbb{E}^{\P}\left[\int_0^T f(s,x_s) \d s\right] = \mathbb{E}^{\tilde{\P}}\left[\int_0^T f(s,x_s) \d s\right],$$
which is sufficient to prove uniqueness in law (see e.g. \cite{EK86}).\\
\end{paragraph}	

Since $X_t^m = x_t^m$, $p^m$ is the density of the canonical process under $\P^m$.  From the Arzelà-Ascoli theorem which can be applied from the estimates derived in Theorem \ref{mol-thm}, we can extract a subsequence $(p^{m_k},\nabla_x p^{m_k})_k$ s.t. $(p^{m_k})_k$ and $(\nabla_x p^{m_k})_k$ converge uniformly on every compact subset to some functions $p$ and $\nabla_x p$ ($\nabla_x p$ being the derivative of $p$). By the uniqueness results from \cite{CdRM22}, $p$ is the time marginal of $\P$, and enjoys the estimates of Theorem \ref{thm-main}.

\section{Proofs}\label{sec-proof}

\begin{proof}[Proof of Lemma \ref{pre-big-lemma} (Bounds and $L^p$ estimates for the noise density)]
The first item about spatial moments (\ref{spatial-moments}) is plain from the definition. The second item \eqref{space-time-derivatives-bounds} stating bounds for space-time derivatives of $p_\alpha$ is proved in \cite{CdRM22}, although in a more general setting (see also \cite{Kol00}, Proposition 2.5 for the current absolutely continuous setting).\\

\noindent For the distortion part (\ref{disto}), if $\ell ' \neq \infty$, it suffices to write
\begin{equation*}
\Vert \bar{p}_\alpha (v,\cdot) \Vert_{L^{\ell'}}^{\ell'} = \int \frac{1}{v^{\frac{d\ell'}{\alpha}}} \frac{1}{\left( 1+\frac{|x|}{v^{\frac{1}{\alpha}}}\right)^{\ell'(d+\alpha)}} \d x = v^{-\frac{d}{\alpha}(\ell'-1)}\int \frac{1}{v^{\frac{d}{\alpha}}} \frac{1}{\left( 1+\frac{|x|}{v^{\frac{1}{\alpha}}}\right)^{\ell'(d+\alpha)}} \d x \lesssim  v^{-\frac{d}{\alpha}(\ell'-1)}.
\end{equation*}
For $\ell = +\infty$, the result is plain from the definition of $\bar{p}_\alpha$.\\
Let us now prove the convolution part (\ref{p-q-convo}). Denote
\begin{align*}
\mathfrak{I} &:= \Vert \bar{p}_\alpha (t-u,\cdot-y) \bar{p}_\alpha (u-s,x-\cdot)\Vert_{L^{\ell'}}^{\ell'}\\
&\lesssim \int \frac{1}{(t-u)^{\frac{d\ell'}{\alpha}}} \frac{1}{\left( 1+\frac{|z-y|}{(t-u)^{\frac{1}{\alpha}}}\right)^{(d+\alpha)\ell'}} \frac{1}{(u-s)^{\frac{d\ell'}{\alpha}}} \frac{1}{\left( 1+\frac{|x-z|}{(u-s)^{\frac{1}{\alpha}}}\right)^{(d+\alpha)\ell'}}\d z.
\end{align*}

\begin{itemize}
\item \textbf{Diagonal case:} $|x-y|<(t-s)^{1/\alpha}$\\
In this case, either $(t-u)\geq \frac{1}{2} (t-s)$ or $(u-s)\geq \frac{1}{2} (t-s)$, i.e. one of the two contributions in $\mathfrak{I}$ is in the diagonal regime, allowing us to use \eqref{diag}.
\begin{itemize}
\item If $(t-u)\geq \frac{1}{2} (t-s)$,
\begin{align*}
\mathfrak{I} &\lesssim \int \frac{1}{(t-u)^{\frac{d\ell'}{\alpha}}} \frac{1}{\left( 1+\frac{|z-y|}{(t-u)^{\frac{1}{\alpha}}}\right)^{(d+\alpha)\ell'}} \frac{1}{(u-s)^{\frac{d\ell'}{\alpha}}} \frac{1}{\left( 1+\frac{|x-z|}{(u-s)^{\frac{1}{\alpha}}}\right)^{(d+\alpha)\ell'}}\d z\\
&\lesssim \frac{1}{(t-u)^{\frac{d\ell'}{\alpha}}}  \frac{1}{(u-s)^{\frac{d}{\alpha}(\ell'-1)}} \int \frac{1}{(u-s)^{\frac{d}{\alpha}}} \frac{1}{\left( 1+\frac{|x-z|}{(u-s)^{\frac{1}{\alpha}}}\right)^{(d+\alpha)\ell'}}\d z\\
&\lesssim \frac{1}{(t-u)^{\frac{d\ell'}{\alpha}}}  \frac{1}{(u-s)^{\frac{d}{\alpha}(\ell'-1)}}.
\end{align*}
Since $(t-u)\geq \frac{1}{2} (t-s)$,
$$\frac{1}{(t-u)^{\frac{d}{\alpha}}} \lesssim \frac{1}{(t-s)^{\frac{d}{\alpha}}} \asymp \bar{p}_\alpha  (t-s,y-x),$$
hence
\begin{equation*}
\mathfrak{I}\lesssim \bar{p}_\alpha (t-s,x-y)^{\ell'} \frac{1}{(u-s)^{\frac{d}{\alpha}(\ell'-1)}}.
\end{equation*}
\item If $(u-s)\geq \frac{1}{2} (t-s)$, the same computations give, when swapping the roles of $u-s$ and $t-u$,
\begin{equation*}
\mathfrak{I}\lesssim \bar{p}_\alpha (t-s,x-y)^{\ell'} \frac{1}{(t-u)^{\frac{d}{\alpha}(\ell'-1)}}.
\end{equation*}
\end{itemize}
\item \textbf{Off-diagonal case:} $|x-y|\geq (t-s)^{1/\alpha}$\\
In this case, either $|x-z|\geq\frac{1}{2}|x-y|$ or $|z-y|>\frac{1}{2}|x-y|$, i.e. one of the two contributions in $\mathfrak{I}$ is in the off-diagonal regime, allowing us to use \eqref{off-diag}.
\begin{itemize}
\item If $|x-z|\geq\frac{1}{2}|x-y|>\frac{1}{2} (t-s)^{1/\alpha}$,
\begin{align*}
\mathfrak{I} &\lesssim \frac{1}{(u-s)^{\frac{d\ell'}{\alpha}}} \frac{1}{\left( 1+\frac{|x-y|}{(u-s)^{\frac{1}{\alpha}}}\right)^{(d+\alpha)\ell'}} \int \frac{1}{(t-u)^{\frac{d\ell'}{\alpha}}} \frac{1}{\left( 1+\frac{|z-y|}{(t-u)^{\frac{1}{\alpha}}}\right)^{(d+\alpha)\ell'}} \mathbb{1}_{|x-z|\geq\frac{1}{2}|x-y|} \d z\\
&\lesssim \bar{p}_\alpha (u-s,x-y)^{\ell'} \frac{1}{(t-u)^{\frac{d}{\alpha}(\ell'-1)}}\int \frac{1}{(t-u)^{\frac{d}{\alpha}}} \frac{1}{\left( 1+\frac{|z-y|}{(t-u)^{\frac{1}{\alpha}}}\right)^{(d+\alpha)\ell'}}\mathbb{1}_{|x-z|\geq\frac{1}{2}|x-y|} \d z 
\end{align*}

Since $|x-y|>(u-s)^{1/\alpha}$,
$$\bar{p}_\alpha (u-s,x-y) \asymp \frac{u-s}{|x-y|^{d+\alpha}} \leq \frac{t-s}{|x-y|^{d+\alpha}} \asymp \bar{p}_\alpha (t-s,x-y),$$
hence
\begin{equation*}
\mathfrak{I}\lesssim \bar{p}_\alpha (t-s,x-y)^{\ell'} \frac{1}{(t-u)^{\frac{d}{\alpha}(\ell'-1)}}.
\end{equation*}
\begin{equation*}
	\mathfrak{I}\lesssim \bar{p}_\alpha (t-s,x-y)^{\ell'} \frac{1}{(t-u)^{\frac{d}{\alpha}(\ell'-1)}}\int \frac{1}{(t-u)^{\frac{d}{\alpha}}} \frac{1}{\left( 1+\frac{|z-y|}{(t-u)^{\frac{1}{\alpha}}}\right)^{(d+\alpha)\ell'}}\mathbb{1}_{|x-z|\geq\frac{1}{2}|x-y|} \d z \\
\end{equation*}
\item If $|z-y|>\frac{1}{2}|x-y|>\frac{1}{2} (t-s)^{1/\alpha}$, the same computations give, when swapping the roles of $|x-z|$ and $|y-z|$,
\begin{equation}
\mathfrak{I}\lesssim \bar{p}_\alpha (t-s,x-y)^{\ell'} \frac{1}{(u-s)^{\frac{d}{\alpha}(\ell'-1)}}.
\end{equation}
\end{itemize}
\end{itemize}

In each case, we have 
\begin{align*}
\Vert \bar{p}_\alpha (t-u,\cdot-y) \bar{p}_\alpha (u-s,x-\cdot)\Vert_{L^{\ell'}} &= \mathfrak{I}^{\frac{1}{\ell'}} \lesssim \left[ \frac{1}{(t-u)^{\frac{d}{\alpha}\frac{\ell'-1}{\ell'}}} +\frac{1}{(u-s)^{\frac{d}{\alpha}\frac{\ell'-1}{\ell'}}}\right] \bar{p}_\alpha (t-s,x-y) \nonumber \\
& \lesssim \left[ \frac{1}{(t-u)^{\frac{d}{\alpha  \ell}}} +\frac{1}{(u-s)^{\frac{d}{\alpha  \ell}}}\right] \bar{p}_\alpha(t-s,x-y).
\end{align*}
Which concludes the proof of Lemma \ref{pre-big-lemma}.
\end{proof}

\begin{proof}[Proof of Lemma \ref{holder-noise} (Taylor-Laplace for stable densities)] We will only prove (\ref{taylor-stable-density}) as (\ref{taylor-approx-density}) follows from the same proof.\\ Since we are working in the diagonal case $|z-w|\lesssim t^{1/\alpha}$, it makes sense to use a Taylor formula:
\begin{align*}
	|D^\ell p_\alpha (t,x-w) - D^\ell p_\alpha (t,x-z)| &\leq \int_0^1 \left| D^{\ell+1}p_\alpha (t,x+\lambda (z-w)-z)\cdot (z-w) \right| \d \lambda\\
	&\lesssim \frac{|z-w|}{t^{\frac{\ell+1}{\alpha}}} \int_0^1 \left| \bar{p}_\alpha (t,x+\lambda (z-w)-z) \right| \d \lambda\\
	&\lesssim \frac{|z-w|^\zeta }{t^{\frac{\ell+\zeta}{\alpha}}} \int_0^1  \bar{p}_\alpha (t,x+\lambda (z-w)-z) \d \lambda,
\end{align*}
where, in the last inequality, we introduced a free exponent $\zeta \in (0,1]$.\\
Notice that $|w-x|\leq |x+\lambda (z-w) -z| + (1-\lambda)|z-w|$. We therefore deduce that in the current diagonal regime, $\frac{|w-x|}{t^{1/\alpha}}\lesssim \frac{|x+\lambda (z-w) -z| }{t^{1/\alpha}} + (1-\lambda)$, hence
$$\bar{p}_\alpha (t,w-x)  \gtrsim t^{-\frac{d}{\alpha}} \left( 1+(1-\lambda) + \frac{|x+\lambda (z-w)-z|}{t^{\frac{1}{\alpha}}}\right)^{-d-\alpha} \gtrsim t^{-\frac{d}{\alpha}} \left( 1 + \frac{|x+\lambda (z-w)-z|}{t^{\frac{1}{\alpha}}}\right)^{-d-\alpha},$$
which in turn yields $$ \int_0^1  \bar{p}_\alpha (t,x+\lambda (z-w)-z) \d \lambda \lesssim \bar{p}_\alpha (t,w-x), $$
thus concluding the proof.
\end{proof}

\begin{proof}[Proof of Lemma \ref{big-lemma} (Besov estimates for $\bar{p}_\alpha$)]
Let's begin with the proof of item (\ref{big-lemma-1}), of which \eqref{big-lemma-2} is a consequence. We will use the thermic characterization of the Besov norm introduced in section \ref{sec-lemmas}:

\begin{equation*}
\Vert \nabla^{j} \bar{p}_{\alpha} (u-s,\cdot-x) \nabla^k p_\alpha (t-u,y-\cdot) \Vert_{\B^{-\beta}_{p',q'}}  = \Vert \phi(D) \mathfrak{q} \Vert_{L^{p'}} + \mathcal{T}_{p',q'}^{-\beta} [\mathfrak{q}(\cdot)],
\end{equation*}
where $\mathfrak{q}(\cdot):=\mathfrak{q}_{s,x,t,u,y}(\cdot)=\nabla^{j}\bar{p}_\alpha (u-s,x-\cdot) \nabla^k p_\alpha (t-u,\cdot-y)$.
\begin{paragraph}{Thermic part \\[0.5cm]}

Recall the definition of the thermic part:

\begin{align*}
\mathcal{T}_{p',q'}^{-\beta} [\mathfrak{q}] ^{q'} &= \int_0^1 \frac{\d v}{v} v^{\left(1+\frac{\beta}{\alpha}\right)q'} \Vert \partial_v p_\alpha (v,\cdot) \star \mathfrak{q} (\cdot)\Vert_{L^{p'}}^{q'}\\
&= \int_0^1 \frac{\d v}{v} v^{\left(1+\frac{\beta}{\alpha}\right)q'} \Vert \partial_v p_\alpha (v,\cdot) \star \nabla^{j}\bar{p}_{\alpha} (u-s,x-\cdot) \nabla^k p_\alpha (t-u,\cdot-y) \Vert_{L^{p'}}^{q'},
\end{align*}
and let us now split this integral over $(0,1)$ into two parts $(0,a)$ and $(a,1)$, where $a$ is a real number in $(0,1)$ to be specified later, which will allow us to balance the contributions of each part of the integral. In the upper part, as $v$ is distinct from 0, no singularities appear and we can simply use convolution inequalities and Lemma \ref{deriv-density-lemma}.
\begin{align}\label{upper-cut} 
\nonumber\mathcal{T}_{p',q'}^{-\beta,(a,1)} [\mathfrak{q}]  &:= \int_a^1 \frac{\d v}{v} v^{\left(1+\frac{\beta}{\alpha}\right)q'} \Vert \partial_v p_\alpha (v,\cdot) \star \nabla^{j}\bar{p}_{\alpha} (u-s,x-\cdot) \nabla^k p_\alpha (t-u,\cdot-y) \Vert_{L^{p'}}^{q'} \\
&\nonumber\lesssim \int_a^1 \frac{\d v}{v} v^{\left(1+\frac{\beta}{\alpha}\right)q'} \left\Vert \frac{1}{v} \bar{p}_\alpha (v,\cdot) \star \frac{1}{(u-s)^{\frac{j}{\alpha}}}\bar{p}_{\alpha} (u-s,x-\cdot) \frac{1}{(t-u)^{\frac{k}{\alpha}}}\bar{p}_\alpha (t-u,\cdot-y) \right\Vert_{L^{p'}}^{q'}\\
&\nonumber\lesssim \frac{1}{(u-s)^{\frac{jq'}{\alpha}}(t-u)^{\frac{kq'}{\alpha}}} \int_a^1  v^{\left(1+\frac{\beta}{\alpha}-1\right)q'-1} \left\Vert \bar{p}_\alpha (v,\cdot) \right\Vert_{L^{1}}^{q'} \left\Vert \bar{p}_{\alpha} (u-s,x-\cdot)  \bar{p}_\alpha (t-u,\cdot-y) \right\Vert_{L^{p'}}^{q'}\d v\\
&\nonumber\lesssim \frac{1}{(u-s)^{\frac{jq'}{\alpha}}(t-u)^{\frac{kq'}{\alpha}}} \int_a^1  v^{\frac{\beta q'}{\alpha}-1}\left[ \frac{1}{(t-u)^{\frac{dq'}{\alpha  p}}} +\frac{1}{(u-s)^{\frac{dq'}{\alpha  p}}}\right] \bar{p}_\alpha^{q'} (t-s,x-y)\d v\\
&\lesssim \frac{\bar{p}_\alpha^{q'}(t-s,x-y)}{(u-s)^{\frac{jq'}{\alpha}}(t-u)^{\frac{kq'}{\alpha}}}\left[ \frac{1}{(t-u)^{\frac{dq'}{\alpha  p}}} +\frac{1}{(u-s)^{\frac{dq'}{\alpha  p}}}\right] a^{\frac{\beta q'}{\alpha}}.
\end{align}
\begin{remark}
	Pay attention to the fact that $\mathcal{T}_{p',q'}^{-\beta,(a,1)}$ is not homogeneous to $\mathcal{T}_{p',q'}^{-\beta} [\mathfrak{q}]$ as we omitted the exponent $q'$ in order to avoid Young inequalities down the line.
\end{remark}	

Now with the lower part:
\begin{align}\label{lower-cut}
\mathcal{T}_{p',q'}^{-\beta,(0,a)} [\mathfrak{q}] &:= \int_0^a \frac{\d v}{v} v^{\left(1+\frac{\beta}{\alpha}\right)q'} \Vert \partial_v p_\alpha (v,\cdot) \star \nabla^j \bar{p}_{\alpha} (u-s,x-\cdot) \nabla^k p_\alpha (t-u,\cdot-y) \Vert_{L^{p'}}^{q'}.
\end{align}

Let us write
\begin{align}\label{elepe-norm-th}
\Vert & \partial_v p_\alpha (v,\cdot) \star \nabla^j \bar{p}_{\alpha} (u-s,x-\cdot) \nabla^k p_\alpha (t-u,\cdot-y) \Vert_{L^{p'}}^{q'} \nonumber \\
&= \left(\int\left|\int \partial_v p_\alpha (v,z-w)  \nabla^j \bar{p}_{\alpha} (u-s,x-w) \nabla^k p_\alpha (t-u,w-y) \d w \right|^{p'}\d z \right)^{\frac{q'}{p'}} \nonumber \\
&= \left(\int\left|\int \partial_v p_\alpha (v,z-w) [ \nabla^j\bar{p}_{\alpha} (u-s,x-w) \nabla^k p_\alpha (t-u,w-y) \right. \right. \nonumber \\& \qquad \qquad\qquad\qquad\qquad\qquad\qquad \left.\left. - \nabla^j\bar{p}_{\alpha} (u-s,x-z) \nabla^k p_\alpha (t-u,z-y) ] \d w \right|^{p'}\d z \right)^{\frac{q'}{p'}},
\end{align}

where, for the last equality, we have used the cancellation argument:
$$ \int \partial_v p_\alpha (v,z-w) \nabla^j\bar{p}_{\alpha} (u-s,x-z) \nabla^k p_\alpha (t-u,z-y) \d w =0.$$
The point of this cancellation is to compensate singularities in $v^{-1+\frac{\beta}{\alpha}q'}$ through a control involving $|z-w|^\zeta$ for some $\zeta>\beta$ and the regularization properties of the thermic kernel, (\ref{spatial-moments}). Next, let us distinguish whether this difference is in diagonal or off-diagonal regime.
\begin{itemize}
	\item \textbf{Diagonal case: $|z-w|\leq (u-s)^\frac{1}{\alpha}$}. Let us write
	\begin{align*}
		&|\nabla^j\bar{p}_{\alpha} (u-s,x-w) \nabla^k p_\alpha (t-u,w-y) - \nabla^j\bar{p}_{\alpha} (u-s,x-z) \nabla^k p_\alpha (t-u,z-y)|\\
		& \lesssim |\nabla^j\bar{p}_{\alpha} (u-s,x-w)  - \nabla^j\bar{p}_{\alpha} (u-s,x-z)||\nabla^k p_\alpha (t-u,w-y)| \nonumber \\ &\qquad + |\nabla^j\bar{p}_{\alpha} (u-s,x-z)||\nabla^k p_\alpha (t-u,w-y)- \nabla^k p_{\alpha} (t-u,z-y) |.
	\end{align*}
	Using \eqref{taylor-approx-density} and the fact that in the diagonal case, $\bar{p}_\alpha(u-s,x-w)\asymp \bar{p}_\alpha(u-s,x-z)$, we have, for any $\zeta \in (0,1]$,
	\begin{align}
		&|\nabla^j\bar{p}_{\alpha} (u-s,x-w) \nabla^k p_\alpha (t-u,w-y) -\nonumber \nabla^j\bar{p}_{\alpha} (u-s,x-z) \nabla^k p_\alpha (t-u,z-y)|\\
		& \lesssim \frac{|z-w|^\zeta}{(u-s)^{\frac{\zeta+j}{\alpha}}(t-u)^{\frac{k}{\alpha}}}  \bar{p}_{\alpha} (u-s,x-w)\bar p_\alpha (t-u,w-y) \nonumber \\ &\qquad \nonumber + \frac{|z-w|^\zeta}{(u-s)^{\frac{j}{\alpha}}(t-u)^{\frac{\zeta+k}{\alpha}}} \bar{p}_{\alpha} (u-s,x-z)\left(\bar p_\alpha (t-u,w-y)+  \bar p_{\alpha} (t-u,z-y) \right)\\
		& \lesssim \frac{|z-w|^\zeta}{(u-s)^{\frac{j}{\alpha}}(t-u)^{\frac{k}{\alpha}}}\nonumber \left[\frac{1}{(u-s)^{\frac{\zeta}{\alpha}}}+\frac{1}{(t-u)^{\frac{\zeta}{\alpha}}}\right] \bar{p}_{\alpha} (u-s,x-w)\bar p_\alpha (t-u,w-y) \nonumber \\
		&\qquad + \frac{|z-w|^\zeta}{(u-s)^{\frac{j}{\alpha}}(t-u)^{\frac{\zeta+k}{\alpha}}} \bar{p}_{\alpha} (u-s,x-z)\bar p_{\alpha} (t-u,z-y).\label{pivot-diag}
	\end{align}
	\item \textbf{Diagonal case: $|z-w|\geq (u-s)^\frac{1}{\alpha}$}. In that case, we simply use a triangular inequality along with \eqref{space-time-derivatives-bounds} and multiply by $(|z-w|/(u-s)^{1/\alpha})^\zeta\geq 1$:
	\begin{align}
		&|\nabla^j\bar{p}_{\alpha} (u-s,x-w) \nabla^k p_\alpha (t-u,w-y) - \nabla^j\bar{p}_{\alpha} (u-s,x-z) \nabla^k p_\alpha (t-u,z-y)|\nonumber\\
		& \lesssim \frac{|z-w|^\zeta}{(u-s)^{\frac{j+\zeta}{\alpha}}(t-u)^{\frac{k}{\alpha}}} \left( \bar{p}_{\alpha} (u-s,x-w)\bar p_\alpha (t-u,w-y)+\bar{p}_{\alpha} (u-s,x-z)\bar p_{\alpha} (t-u,z-y)\right) .\label{pivot-off-diag}
	\end{align}
\end{itemize}
Plugging \eqref{pivot-diag} and \eqref{pivot-off-diag} into \eqref{elepe-norm-th}, we get
\begin{align*}
	\Vert & \partial_v p_\alpha (v,\cdot) \star \nabla^j\bar{p}_{\alpha} (u-s,x-\cdot) \nabla^k p_\alpha (t-u,\cdot-y) \Vert_{L^{p'}}^{q'} \\
	&\lesssim  \left(\int\left(\int \frac{ |\partial_v p_\alpha (v,z-w)||z-w|^\zeta}{(u-s)^{\frac{j}{\alpha}}(t-u)^{\frac{k}{\alpha}}}\nonumber \left[\frac{1}{(u-s)^{\frac{\zeta}{\alpha}}}+\frac{1}{(t-u)^{\frac{\zeta}{\alpha}}}\right] \bar{p}_{\alpha} (u-s,x-w)\bar p_\alpha (t-u,w-y)  \d w \right)^{p'}\d z \right)^{\frac{q'}{p'}}\\
	&\qquad +  \left(\int\left(\int  \frac{|\partial_v p_\alpha (v,z-w)||z-w|^\zeta}{(u-s)^{\frac{j}{\alpha}}(t-u)^{\frac{\zeta+k}{\alpha}}} \bar{p}_{\alpha} (u-s,x-z)\bar p_{\alpha} (t-u,z-y) \d w \right)^{p'}\d z \right)^{\frac{q'}{p'}}\\
	&=: I_1+I_2.
\end{align*}
For $I_2$, we can immediately derive a smoothing effect in $v$ by using the moments estimate \eqref{spatial-moments} of the stable kernel. Namely, this yields
\begin{align*}
	 I_2 & \lesssim \left(\frac{v^{\frac{\zeta}{\alpha}-1}}{(u-s)^{\frac{j}{\alpha}}(t-u)^{\frac{\zeta+k}{\alpha}}}\right)^{q'} \Vert \bar{p}_{\alpha} (u-s,x-\cdot)\bar p_{\alpha} (t-u,\cdot-y) \Vert_{L^{p'}}^{q'}\\
	 &\lesssim \left(\frac{v^{\frac{\zeta}{\alpha}-1}}{(u-s)^{\frac{j}{\alpha}}(t-u)^{\frac{\zeta+k}{\alpha}}}\right)^{q'}\left[\frac{1}{(t-u)^{\frac{dq'}{\alpha p}}} + \frac{1}{(t-s)^{\frac{dq'}{\alpha p}}}\right] \bar{p}_\alpha^{q'}(t-s,y-x),
\end{align*}
where we used \eqref{p-q-convo} for the last inequality. For $I_1$, due to the integration order, we first have to use an $L^1-L^{p'}$ convolution inequality. Proceeding similarly, we get
\begin{align*}
	 I_1 \lesssim \left(\frac{v^{\frac{\zeta}{\alpha}-1}}{(u-s)^{\frac{j}{\alpha}}(t-u)^{\frac{k}{\alpha}}}\right)^{q'}\left[\frac{1}{(u-s)^{\frac{\zeta}{\alpha}}}+\frac{1}{(t-u)^{\frac{\zeta}{\alpha}}}\right]^{q'}\left[\frac{1}{(t-u)^{\frac{dq'}{\alpha p}}} + \frac{1}{(t-s)^{\frac{dq'}{\alpha p}}}\right] \bar{p}_\alpha^{q'}(t-s,y-x).
\end{align*}
We thus obtain 
\begin{align*}
	\Vert & \partial_v p_\alpha (v,\cdot) \star \nabla^j\bar{p}_{\alpha} (u-s,x-\cdot) \nabla^k p_\alpha (t-u,\cdot-y) \Vert_{L^{p'}}^{q'} \\
	&\lesssim \left(\frac{v^{\frac{\zeta}{\alpha}-1}}{(u-s)^{\frac{j}{\alpha}}(t-u)^{\frac{k}{\alpha}}}\right)^{q'}\left[\frac{1}{(u-s)^{\frac{\zeta}{\alpha}}}+\frac{1}{(t-u)^{\frac{\zeta}{\alpha}}}\right]^{q'}\left[\frac{1}{(t-u)^{\frac{dq'}{\alpha p}}} + \frac{1}{(t-s)^{\frac{dq'}{\alpha p}}}\right] \bar{p}_\alpha^{q'}(t-s,y-x).
\end{align*}
Going back to (\ref{lower-cut}),
\begin{align*}
	&\mathcal{T}_{p',q'}^{-\beta,(0,a)} [\mathfrak{q}] := \int_0^a \frac{\d v}{v} v^{\left(1+\frac{\beta}{\alpha}\right)q'} \Vert \partial_v p_\alpha (v,\cdot) \star \nabla^j\bar{p}_{\alpha} (u-s,x-\cdot) \nabla^k p_\alpha (t-u,\cdot-y) \Vert_{L^{p'}}^{q'} \\
	&\lesssim  \left(\frac{1}{(u-s)^{\frac{j}{\alpha}}(t-u)^{\frac{k}{\alpha}}}\right)^{q'}\left[\frac{1}{(u-s)^{\frac{\zeta}{\alpha}}}+\frac{1}{(t-u)^{\frac{\zeta}{\alpha}}}\right]^{q'}\left[\frac{1}{(t-u)^{\frac{dq'}{\alpha p}}} + \frac{1}{(t-s)^{\frac{dq'}{\alpha p}}}\right] \bar{p}_\alpha^{q'}(t-s,y-x)\\ & \qquad \qquad \times \int_0^a  \frac{1}{v} v^{\left(1+\frac{ \beta}{\alpha}\right)q'} v^{q'(\frac{\zeta}{\alpha}-1)} \d v\\
	&\lesssim  \left(\frac{1}{(u-s)^{\frac{j}{\alpha}}(t-u)^{\frac{k}{\alpha}}}\right)^{q'}\left[\frac{1}{(u-s)^{\frac{\zeta}{\alpha}}}+\frac{1}{(t-u)^{\frac{\zeta}{\alpha}}}\right]^{q'}\left[\frac{1}{(t-u)^{\frac{dq'}{\alpha p}}} + \frac{1}{(t-s)^{\frac{dq'}{\alpha p}}}\right] \bar{p}_\alpha^{q'}(t-s,y-x)a^{q'\left(\frac{\beta+\zeta}{\alpha}\right)},
\end{align*}
where we set $\zeta > -\beta$ so that the integral converges. Setting the overall value for $a=t-s$ and recalling \eqref{upper-cut}, we have obtained
\begin{align*}
	&\mathcal{T}_{p',q'}^{-\beta} [\mathfrak{q}] =\left( \mathcal{T}_{p',q'}^{-\beta,(0,t-s)} [\mathfrak{q}] + \mathcal{T}_{p',q'}^{-\beta,(t-s,1)} [\mathfrak{q}]\right)^{\frac{1}{q'}}\\
	&\lesssim \left(\frac{1}{(u-s)^{\frac{j}{\alpha}}(t-u)^{\frac{k}{\alpha}}}\right)\left[1+\frac{(t-s)^{\frac{\zeta}{\alpha}}}{(u-s)^{\frac{\zeta}{\alpha}}}+\frac{(t-s)^{\frac{\zeta}{\alpha}}}{(t-u)^{\frac{\zeta}{\alpha}}}\right]\left[\frac{1}{(t-u)^{\frac{d}{\alpha p}}} + \frac{1}{(t-s)^{\frac{d}{\alpha p}}}\right] \bar{p}_\alpha(t-s,y-x)(t-s)^{\frac{\beta}{\alpha}}.
\end{align*}
\end{paragraph}

\begin{paragraph}{Non-thermic part\\[0.5cm]}
	
	For the non-thermic part, let us write
	\begin{align*}
		\Vert \mathcal{F}(\phi)\star \nabla^j\bar{p}_\alpha (u-s,&x-\cdot)\nabla^k p_\alpha (t-u,y-\cdot)\Vert_{L^{p'}} \leq \Vert  \mathcal{F}(\phi)\Vert_{L^{1}} \Vert \nabla^j\bar{p}_\alpha (u-s,x-\cdot)\nabla^k p_\alpha (t-u,y-\cdot)\Vert_{L^{p'}}\\
		&\lesssim \frac{\bar{p}_\alpha (t-s,y-x)}{(u-s)^{\frac{j}{\alpha}}(t-u)^{\frac{k}{\alpha}}}\left[ \frac{1}{(t-u)^{\frac{d}{\alpha  p}}} +\frac{1}{(u-s)^{\frac{d}{\alpha  p}}}\right]\\
		& \lesssim \mathcal{T}_{p',q'}^{-\beta} [\mathfrak{q}].
	\end{align*}
\end{paragraph}
Which concludes the proof of the first item of Lemma \ref{big-lemma}.\\

In order to prove the second item (\ref{big-lemma-2}), we will bound cleverly $\left| \frac{\nabla p_\alpha (t-u,w-z)}{\bar{p}_{\alpha} (t-s,x-w)} - \frac{\nabla p_\alpha (t-u,y-z)}{\bar{p}_{\alpha} (t-s,x-y)} \right|$ and then use the previous result to conclude.\\
Notice that, in (\ref{big-lemma-2}), the time-space scale we are looking at is $(t-u)^{1/\alpha}$ compared to $|w-y|$, hence we will need to treat split the discussion into diagonal and off-diagonal cases at this local time scale. Let $z\in \R^d$.
\begin{itemize}
	\item If $|w-y|\geq (t-u)^{1/\alpha}$, as usual, we use a triangular inequality and the fact that $\frac{|w-y|^\zeta}{(t-u)^{\frac{\zeta}{\alpha}}} \geq 1$.
	
	\begin{align*}
		\left| \frac{\nabla p_\alpha (t-u,w-z)}{\bar{p}_{\alpha} (t-s,w-x)} - \right.&\left.\frac{\nabla p_\alpha (t-u,y-z)}{\bar{p}_{\alpha} (t-s,y-x)} \right| \\ &\lesssim \frac{|w-y|^\zeta}{(t-u)^{\frac{\zeta}{\alpha}}} \left[ \frac{|\nabla p_\alpha (t-u,w-z)|}{\bar{p}_\alpha (t-s,w-x)} + \frac{|\nabla  p_\alpha (t-u,y-z)|}{\bar{p}_\alpha(t-s,y-x)} \right].
	\end{align*}
	\item If $|w-y|\leq (t-u)^{1/\alpha}$, we use the Taylor of expansion of Lemma \ref{holder-noise}:
	\begin{align*}
		\left| \frac{\nabla p_\alpha (t-u,w-z)}{\bar{p}_{\alpha} (t-s,w-x)} \right. &\left. -\frac{\nabla p_\alpha (t-u,y-z)}{\bar{p}_{\alpha} (t-s,y-x)} \right|\\ & = \left| \frac{\nabla p_\alpha (t-u,w-z)}{\bar{p}_{\alpha} (t-s,w-x)} \pm  \frac{\nabla p_\alpha (t-u,w-z)}{\bar{p}_{\alpha} (t-s,y-x)}-\frac{\nabla p_\alpha (t-u,y-z)}{\bar{p}_{\alpha} (t-s,y-x)} \right| \\
		&\lesssim | \nabla p_\alpha (t-u,w-z) | \left|\frac{1}{\bar{p}_{\alpha} (t-s,w-x)} -  \frac{1}{\bar{p}_{\alpha} (t-s,y-x)}\right| \\&\qquad+ \frac{1}{\bar{p}_\alpha (t-s,y-x)} \left| \nabla p_\alpha (t-u,w-z)  - \nabla p_\alpha (t-u,y-z)\right| \\
		&\lesssim  \frac{\bar{p}_\alpha (t-u,w-z)}{(t-u)^{\frac{1}{\alpha}}}  \left|\frac{1}{\bar{p}_{\alpha} (t-s,w-x)} -  \frac{1}{\bar{p}_{\alpha} (t-s,y-x)}\right| \\&\qquad+ \frac{1}{\bar{p}_\alpha (t-s,y-x)} \left| \nabla p_\alpha (t-u,w-z)  - \nabla p_\alpha (t-u,y-z)\right| \\
		& \lesssim \frac{|w-y|^\zeta}{(t-u)^{\frac{\zeta + 1}{\alpha}}} \left[ \frac{\bar{p}_\alpha (t-u,w-z)}{\bar{p}_\alpha (t-s,w-x)}  +\frac{ \bar{p}_\alpha (t-u,y-z)  }{\bar{p}_\alpha (t-s,y-x)} \right].
	\end{align*}
	
\end{itemize}
In both cases,
\begin{align}
	\left| \frac{\nabla p_\alpha (t-u,w-z)}{\bar{p}_{\alpha} (t-s,x-w)} -\frac{\nabla p_\alpha (t-u,y-z)}{\bar{p}_{\alpha} (t-s,x-y)} \right| &\lesssim  \frac{|w-y|^\zeta}{(t-u)^{\frac{\zeta + 1}{\alpha}}} \left[ \frac{\bar{p}_\alpha (t-u,w-z)}{\bar{p}_\alpha (t-s,w-x)}  +\frac{ \bar{p}_\alpha (t-u,y-z)  }{\bar{p}_\alpha (t-s,y-x)} \right],
\end{align}
hence 
\begin{align*}
	&\left\Vert \bar{p}_{\alpha} (u-s,x-\cdot)\left[\frac{\nabla p_\alpha (t-u,w-\cdot)}{\bar{p}_{\alpha} (t-s,w-x)} -\frac{\nabla p_\alpha (t-u,y-\cdot)}{\bar{p}_{\alpha} (t-s,y-x)} \right] \right\Vert_{ \B_{p',q'}^{-\beta}}\\& \lesssim \frac{|w-y|^\zeta}{(t-u)^{\frac{\zeta + 1}{\alpha}}} \left[\left\Vert\bar{p}_{\alpha} (u-s,x-\cdot) \frac{\bar{p}_\alpha (t-u,w-\cdot)}{\bar{p}_\alpha (t-s,w-x)}\right\Vert_{ \B_{p',q'}^{-\beta}} +\left\Vert \bar{p}_{\alpha} (u-s,x-\cdot)\frac{ \bar{p}_\alpha (t-u,y-\cdot)  }{\bar{p}_\alpha (t-s,y-x)}\right\Vert_{ \B_{p',q'}^{-\beta}}\right].
\end{align*}
(\ref{big-lemma-2}) then follows from using (\ref{big-lemma-1}) twice.


\end{proof}

\begin{proof}[Proof of Lemma \ref{thermic-h} (Thermic part of $\Vert h_{s,x}(u,\cdot) \Vert_{\B^{\rho}_{\infty,\infty}}$)] \hspace{0.5cm}
	
\noindent Recall the definition of the thermic part (choosing $n=1$)
	$$\mathcal{T}_{\infty,\infty}^\rho [h_{s,x}(t,\cdot)] = \sup_{v\in (0,1]} v^{1-\frac{\rho}{\alpha}} \Vert \partial_v  p_\alpha (v,\cdot) \star h_{s,x}(t,\cdot) \Vert_{L^\infty}.$$
Let $y\in \R^d$. We use the cancellation argument 
$$\int \partial_v p_\alpha (v,y-w)h_{s,x}(t,y)\d w = 0$$
to write
\begin{align*}
| \partial_v  p_\alpha (v,y) \star h_{s,x}(t,y)|& = \left|\int \partial_v p_\alpha (v,y-w)h_{s,x}(t,w)\d w \right|\\
&=\left|\int \partial_v p_\alpha (v,y-w)(h_{s,x}(t,w)-h_{s,x}(t,y))\d w \right|.
\end{align*}

As in the proof of Lemma \ref{big-lemma}, we will treat the diagonal and off-diagonal cases separately. This time, those regimes are to be considered w.r.t. the time increment $(t-s)$ and the spatial increment $|y-w|$.

\begin{itemize}
\item \textbf{Off-diagonal case:} $|w-y|\geq (t-s)^{1/\alpha}$
\begin{align}\label{off-diag-h}
\mathfrak{P_1}&:= \left|\int \partial_v p_\alpha (v,y-w)(h_{s,x}(t,w)-h_{s,x}(t,y)) \mathbb{1}_{|w-y|\geq (t-s)^{\frac{1}{\alpha}}}\d w \right| \nonumber \\
& \lesssim \int | \partial_v p_\alpha (v,y-w)||h_{s,x}(t,w)-h_{s,x}(t,y)| \mathbb{1}_{|w-y|\geq (t-s)^{\frac{1}{\alpha}}}\d w \nonumber \\
 &\lesssim \int \frac{1}{v} \bar{p}_\alpha (v,y-w)\Vert h_{s,x}(t,\cdot) \Vert_{L^\infty} \frac{|w-y|^{\rho}}{(t-s)^{\frac{\rho}{\alpha}}}\d w \nonumber \\
 &\lesssim \frac{1}{v (t-s)^{\frac{\rho}{\alpha}}}\Vert h_{s,x}(t,\cdot) \Vert_{L^\infty} \int \bar{p}_\alpha (v,y-w) |y-w|^\rho \d w \nonumber \\
  &\lesssim \frac{v^{\frac{\rho}{\alpha}-1}}{(t-s)^{\frac{\rho}{\alpha}}}\Vert h_{s,x}(t,\cdot) \Vert_{L^\infty} \nonumber \\
  &\lesssim \frac{v^{\frac{\rho}{\alpha}-1}}{(t-s)^{\frac{\rho}{\alpha}}} \left(  1 +\int_s^t   \Vert h_{s,x}(u,\cdot) \Vert_{\B^{\rho}_{\infty,\infty}} \Vert b^m (u,\cdot) \Vert_{\B^{\beta}_{p,q}}  \mathfrak{L}(u,s,t,\rho)  \d u\right),
\end{align}
where the last inequality was obtained using (\ref{l-inf-bound}).
\item \textbf{Diagonal case:} $|w-y|\leq (t-s)^{1/\alpha}$. Let
\begin{align}
	\Delta_1 (s,t,x,w,y) &:=\frac{p_\alpha(t-s,w-x)}{\bar{p}_{\alpha}(t-s,w-x)}-\frac{p_\alpha(t-s,y-x)}{\bar{p}_{\alpha}(t-s,y-x)} \label{def-delta1}\\
	\Delta_2 (s,t,x,w,y) &:=\int_s^t \int \frac{p^m (u-s,z-x)}{\bar{p}_{\alpha} (u-s,z-x)} b^m (u,z) \bar{p}_{\alpha} (s,u,x,z) \nonumber \\ &\qquad \qquad \times \left[\frac{\nabla p_\alpha (t-u,w-z)}{\bar{p}_{\alpha} (t-s,w-x)} -\frac{\nabla p_\alpha (t-u,y-z)}{\bar{p}_{\alpha} (t-s,y-x)} \right]\d z \d u,\label{def-delta2}
\end{align}

so that $|h_{s,x}(t,w)-h_{s,x}(t,y)| = |	\Delta_1 (s,t,x,w,y)+	\Delta_2 (s,t,x,w,y)|$. Let us first bound $\Delta_1$:
\begin{align*}
	|\Delta_1 |&\leq \left| \frac{p_{\alpha}(t-s,w-x) - p_\alpha (t-s,y-x)}{\bar{p}_\alpha (t-s,w-x)} \right| + p_\alpha (t-s,y-x) \left| \frac{1}{\bar{p}_\alpha (t-s,w-x)} - \frac{1}{\bar{p}_\alpha (t-s,y-x)}\right|\\
	&\leq \left| \frac{p_{\alpha}(t-s,w-x) - p_\alpha (t-s,y-x)}{\bar{p}_\alpha (t-s,w-x)} \right| +  \left| \frac{\bar{p}_\alpha (t-s,y-x)-\bar{p}_\alpha (t-s,w-x)}{\bar{p}_\alpha (t-s,w-x)}\right|.
\end{align*}
We now use Lemma \ref{holder-noise} twice to obtain, $\forall \rho \in (0,1]$,
\begin{equation}\label{delta-1-maj}
	|\Delta_1 (s,t,x,w,y)|\lesssim  \frac{|w-y|^\rho}{(t-s)^{\frac{\rho}{\alpha}}}.
\end{equation}

Let us now treat $\Delta_2$. Using the same inequalities as in Section \ref{sec-estimates}, we can write
\begin{align*}
|\Delta_2 (s,t,x,w,y) |&:=\int_s^t  \Vert h_{s,x}(u,\cdot) \Vert_{ \B_{\infty,\infty}^\rho} \Vert b^m (u,\cdot)\Vert_{ \B_{p,q}^\beta}  \nonumber \\ &\qquad \qquad \times \left\Vert \bar{p}_{\alpha} (u-s,x-\cdot)\left[\frac{\nabla p_\alpha (t-u,w-\cdot)}{\bar{p}_{\alpha} (t-s,w-x)} -\frac{\nabla p_\alpha (t-u,y-\cdot)}{\bar{p}_{\alpha} (t-s,y-x)} \right] \right\Vert_{ \B_{p',q'}^{-\beta}}  \d u.
\end{align*}
Using now (\ref{big-lemma-2}), we obtain
\begin{align}\label{delta-2-maj}
	|\Delta_2 (s,t,x,w,y)| \lesssim |w-y|^\rho \int_s^t  \Vert h_{s,x}(u,\cdot) \Vert_{\B_{\infty,\infty}^\rho}& \Vert b^m(u,\cdot) \Vert_{\B_{p,q}^\beta}  \frac{\mathfrak{L}(u,s,t,\rho)}{(t-u)^{\frac{\rho}{\alpha}}}\d u.
\end{align}

We can now use the bounds from (\ref{delta-1-maj}) and (\ref{delta-2-maj}) to write
\begin{align}
	\mathfrak{P_2}&:= \left|\int \partial_v p_\alpha (v,y-w)(h_{s,x}(t,w)-h_{s,x}(t,y)) \mathbb{1}_{|w-y|< (t-s)^{\frac{1}{\alpha}}}\d w \right| \nonumber \\
	& \lesssim \nonumber \int | \partial_v p_\alpha (v,y-w)||\Delta_1 (s,t,x,w,y)| \mathbb{1}_{|w-y|< (t-s)^{\frac{1}{\alpha}}}\d w \\&  \qquad+  \int | \partial_v p_\alpha (v,y-w)||\Delta_2 (s,t,x,w,y)| \mathbb{1}_{|w-y|< (t-s)^{\frac{1}{\alpha}}}\d w \nonumber\\
	&=: \mathfrak{D}_1 + \mathfrak{D}_2.
\end{align}
Let us start with $\mathfrak{D}_1 $:
\begin{align}
	\mathfrak{D}_1 &\lesssim  \int | \partial_v p_\alpha (v,y-w)| \frac{|w-y|^\rho}{(t-s)^{\frac{\rho}{\alpha}}}\d w \nonumber \\
	&\lesssim \frac{v^{-1}}{(t-s)^{\frac{\rho}{\alpha}}}\int \bar{p}_\alpha (v,y-w) |w-y|^\rho \d w \nonumber  \\
	&\lesssim \frac{v^{\frac{\rho}{\alpha}-1}}{(t-s)^{\frac{\rho}{\alpha}}}.
\end{align}
Now, for $\mathfrak{D}_2$, we integrate in $w$ first:
\begin{align*}
	\mathfrak{D}_2 &= \int | \partial_v p_\alpha (v,y-w)||\Delta_2 (s,t,x,w,y)|\mathbb{1}_{|w-y|< (t-s)^{\frac{1}{\alpha}}}\d w \nonumber \\
	&\lesssim v^{-1} \int_s^t  \Vert h_{s,x}(u,\cdot) \Vert_{\B_{\infty,\infty}^\rho} \Vert b^m(u,\cdot) \Vert_{\B_{p,q}^\beta}  \frac{\mathfrak{L}(u,s,t,\rho)}{(t-u)^{\frac{\rho}{\alpha}}}  \int | \bar{p}_\alpha (v,y-w)||w-y|^\rho\d w \d u \nonumber \\
	&  \lesssim \frac{v^{\frac{\rho}{\alpha}-1}}{(t-s)^{\frac{\rho}{\alpha}}} \int_s^t  \Vert h_{s,x}(u,\cdot) \Vert_{\B_{\infty,\infty}^\rho} \Vert b^m(u,\cdot) \Vert_{\B_{p,q}^\beta}  \mathfrak{L}(u,s,t,\rho) \frac{(t-s)^{\frac{\rho}{\alpha}}}{(t-u)^{\frac{\rho}{\alpha}}}   \d u,
\end{align*}
hence, for the diagonal case $|w-y|\leq (t-s)^{1/\alpha}$, we obtain
\begin{align}\label{diag-h}
	\mathfrak{P}_2 \lesssim \frac{v^{\frac{\rho}{\alpha}-1}}{(t-s)^{\frac{\rho}{\alpha}}} \left(1 + \int_s^t  \Vert h_{s,x}(u,\cdot) \Vert_{\B_{\infty,\infty}^\rho} \Vert b^m(u,\cdot) \Vert_{\B_{p,q}^\beta}  \mathfrak{L}(u,s,t,\rho) \frac{(t-s)^{\frac{\rho}{\alpha}}}{(t-u)^{\frac{\rho}{\alpha}}}   \d u\right).
\end{align}
\end{itemize}
Combining (\ref{off-diag-h}) and (\ref{diag-h}), we get the overall bound
\begin{equation*}
		| \partial_v  p_\alpha (v,y) \star h_{s,x}(t,y)| \lesssim  \frac{v^{\frac{\rho}{\alpha}-1}}{(t-s)^{\frac{\rho}{\alpha}}} \left(1 + \int_s^t  \Vert h_{s,x}(u,\cdot) \Vert_{\B_{\infty,\infty}^\rho} \Vert b^m(u,\cdot) \Vert_{\B_{p,q}^\beta}  \mathfrak{L}(u,s,t,\rho) \left[\frac{(t-s)^{\frac{\rho}{\alpha}}}{(t-u)^{\frac{\rho}{\alpha}}} +1 \right]  \d u\right),
\end{equation*}
hence 
\begin{equation*}
	\mathcal{T}_{\infty,\infty}^\rho [h_{s,x}(t,\cdot)] \lesssim \frac{1}{(t-s)^{\frac{\rho}{\alpha}}} \left(1 + \int_s^t  \Vert h_{s,x}(u,\cdot) \Vert_{\B_{\infty,\infty}^\rho} \Vert b^m(u,\cdot) \Vert_{\B_{p,q}^\beta}  \mathfrak{L}(u,s,t,\rho) \left[\frac{(t-s)^{\frac{\rho}{\alpha}}}{(t-u)^{\frac{\rho}{\alpha}}} +1 \right]  \d u\right).
\end{equation*}
\end{proof}
 
 \large \textbf{Declarations} \\
 \normalsize

 \begin{paragraph}{Ethical approval}Not applicable\\
 \end{paragraph}
 
 \begin{paragraph}{Competing interests}The author has no competing interests to declare that are relevant to the content of this article\\
 \end{paragraph}
 \begin{paragraph}{Authors' contributions}Not applicable\\
 \end{paragraph}
 \begin{paragraph}{Funding} Research supported by the Universit\'e Paris Saclay\\
 \end{paragraph}
 \begin{paragraph}{Availability of data and materials}Not applicable\\
 \end{paragraph}
\bibliographystyle{alpha}
\bibliography{ar1}

\end{document}